\documentclass[review,onefignum,onetabnum]{siamart171218}

\nolinenumbers

\usepackage{lipsum}
\usepackage{amsfonts}
\usepackage{graphicx}
\usepackage{epstopdf}
\usepackage{algorithmic}
\usepackage{subfig}
\usepackage[abs]{overpic}

\def \ve{\varepsilon}

\def\longrightharpoonup{\relbar\joinrel\rightharpoonup}
\def\longleftharpoondown{\leftharpoondown\joinrel\relbar}

\def\longrightleftharpoons{
  \mathop{
    \vcenter{
      \hbox{
      \ooalign{
        \raise1pt\hbox{$\longrightharpoonup\joinrel$}\crcr
	  \lower1pt\hbox{$\longleftharpoondown\joinrel$}
	  }
      }
    }
  }
}

\def\cA{\mathcal{A}}
\def\cB{\mathcal{B}}

\def\Xint#1{\mathchoice
{\XXint\displaystyle\textstyle{#1}}%
{\XXint\textstyle\scriptstyle{#1}}%
{\XXint\scriptstyle\scriptscriptstyle{#1}}%
{\XXint\scriptscriptstyle\scriptscriptstyle{#1}}%
\!\int}
\def\XXint#1#2#3{{\setbox0=\hbox{$#1{#2#3}{\int}$ }
\vcenter{\hbox{$#2#3$ }}\kern-.58\wd0}}

\def\dashint{\Xint-}  

\ifpdf
  \DeclareGraphicsExtensions{.eps,.pdf,.png,.jpg}
\else
  \DeclareGraphicsExtensions{.eps}
\fi

\newsiamremark{remark}{Remark}
\newsiamremark{hypothesis}{Hypothesis}
\crefname{hypothesis}{Hypothesis}{Hypotheses}
\newsiamthm{claim}{Claim}

\headers{Nonstandard scaling}{J. King, J. K\"ory, and M. Ptashnyk}

\title{Multiscale analysis of nutrient uptake by plant roots  with sparse distribution of root hairs: Nonstandard scaling \thanks{Submitted to the editors DATE.
\funding{Jakub K\"ory and John King acknowledge funding from FUTUREROOTS
Project (project ID: 294729) between European Research Council and The University of Nottingham.}}}

\author{John King\thanks{School of Mathematical Sciences \& Centre for Plant Integrative Biology, School of Biosciences, University of Nottingham,  Nottingham NG7 2QL, United Kingdom  (\email{john.king@nottingham.ac.uk}).}
\and Jakub K\"ory\thanks{ School of Mathematics \& Statistics, University of Glasgow, University Place, Glasgow~G12~8QQ, United Kingdom
  (\email{jakub.koery@glasgow.ac.uk}).}
\and Mariya Ptashnyk\thanks{School of Mathematical and Computer Sciences, Heriot-Watt University, Edinburgh  EH14 4AL, United Kingdom (\email{m.ptashnyk@hw.ac.uk}).}}

\usepackage{amsopn}

\ifpdf
\hypersetup{
  pdftitle={Multiscale analysis of nutrient uptake by plant roots  with sparse distribution of root hairs: Nonstandard scaling},
  pdfauthor={J. King, J. K\"ory, and M. Ptashnyk}
}
\fi

\makeatletter
\newcommand*{\addFileDependency}[1]{
  \typeout{(#1)}
  \@addtofilelist{#1}
  \IfFileExists{#1}{}{\typeout{No file #1.}}
}
\makeatother
 
\newcommand*{\myexternaldocument}[1]{%
    \externaldocument{#1}%
    \addFileDependency{#1.tex}%
    \addFileDependency{#1.aux}%
}

\myexternaldocument{KKP_supplement}

\begin{document}

\maketitle

\begin{abstract}
 In this paper we undertake a multiscale analysis of nutrient uptake by  plant roots, considering different scale relations between the radius of root hairs and the distance between them.  We combine the method of formal asymptotic expansions and rigorous derivation of macroscopic equations. The former prompt us to study a distinguished limit (which yields a distinct effective equation), allow us to determine higher order correctors and provide motivation for the construction of correctors essential for rigorous derivation of macroscopic equations. In the final section, we validate the results of our asymptotic analysis by direct comparison with full-geometry numerical simulations.
\end{abstract}


\section{Introduction}

An efficient nutrient uptake by plant roots is very important for plant growth and development \cite{barber1995, brady1996}.  Root hairs, the cylindrically-shaped lateral extensions of epidermal cells that increase the surface area of the root system,  play a significant  role in the uptake of nutrients by plant roots \cite{datta2011}. Thus to optimize the nutrient uptake it is important to understand better the impact of root hairs on the uptake processes.
Early phenomenological models describe the effect of root hairs on the nutrient uptake by increasing the radius of roots \cite{passioura1963}.
Microscopic modelling and analysis  of nutrient uptake by root hairs on the scale of a single hair, assuming periodic distribution of hairs and that the distance between them is of the same order as their radius were considered in  \cite{leitner2010_1, Ptashnyk2010b, zygalakis2011}.

In contrast to previous results, in this work we consider a sparse distribution of root hairs, with the radius of root hairs much smaller than the distance between them. We consider two different regimes given by scaling relations between the hair radius and the distance between neighboring hairs. Applying multiscale analysis techniques, we derive  macroscopic equations from the microscopic description by applying both the  method of formal asymptotic expansions and rigorous proofs of convergences of  sequences of solutions of microscopic (full-geometry) problems.  Due to non-standard scale relations between the size of the microscopic structure and the periodicity, the homogenization techniques of two-scale convergence, the periodic unfolding method, $\Gamma$-  or G-convergences, see e.g.\ \cite{Giorgi1973, DalMaso1993,  Murat1997, Nguetseng89},  do not apply directly  and a different approach needs to be developed.
The construction of inner and outer layer approximation problems constitutes the main idea in the derivation of the  macroscopic problems using formal asymptotic expansions. This approach allows us also to obtain equations for higher-order approximations to the macroscopic solutions. To show convergence of solutions of the multiscale (microscopic) problems to those of the corresponding macroscopic problems, we construct appropriate correctors   to pass to the limit in the integrals over the boundaries of the microstructure given by root hairs.
We also compare numerical solutions  of the multiscale problems with solutions of macroscopic problems  and  higher  (first and second) order approximations, derived for different scale-relations between the size of the hairs and the size of the periodicity.

Similar results for elliptic equations and variational inequalities were  obtained in \cite{gomez2015, Jaeger_small,Jaeger_small_2} using the monotonicity of the nonlinear function in the boundary conditions and a variational inequality approach. The construction of  correctors near surfaces of very small holes  was considered in \cite{Cioranescu_thiny, Donato1988} to derive macroscopic equations for linear elliptic  problems with zero Dirichlet  and given Robin boundary conditions.  The extension of the periodic unfolding method to domains with very small holes was introduced  in \cite{cabarrubias2016}   to analyze linear wave and heat equations posed in periodically perforated domains with small holes and Dirichlet conditions on the boundary of the holes. 

The paper is organized as follows. In Section~\ref{sec_1} we formulate a model for nutrient uptake by plant roots and  root hairs.  In Section~\ref{formal_derivation} we derive macroscopic equations and equations for the first- and second-order correctors, for different scale-relations between the radius of root hairs and the distance between them, by using formal asymptotic expansions. The proof of the convergence of a sequence of solutions of the multiscale problem to those of the macroscopic equations via the construction of corresponding microscopic correctors is given in   Section~\ref{convergence}.   The linear and  nonlinear Robin boundary conditions depending on solution of the microscopic problem considered in this manuscript require new ideas in the construction of the corresponding correctors. Numerical simulations of both multiscale and macroscopic problems are presented in Section~\ref{sec_numerics} and we conclude in Section~\ref{sec:discussion} with a brief discussion.

\section{Formulation of the problem}\label{sec_1}
 We consider diffusion of  nutrients in a domain around a plant root and its uptake by  root hairs and through the root surface. The representative length of the root is chosen to be $R=1$~cm and the model is subsequently formulated in dimensionless terms
 (see the Supplementary materials for comments on the non-dimensionalization and on parameter values). The root surface is treated as planar, which approximates the actual (curved) geometry well enough, provided that the distance between hairs measured at the root surface is comparable to the distance between hair tips, as discussed in \cite{leitner2010_1}. A generalization that addresses root curvature is investigated in \cite{kory2018}.

\begin{figure}
\centering
\subfloat[Multiscale domain]
{\begin{overpic}[ width=0.61\textwidth,tics=10]{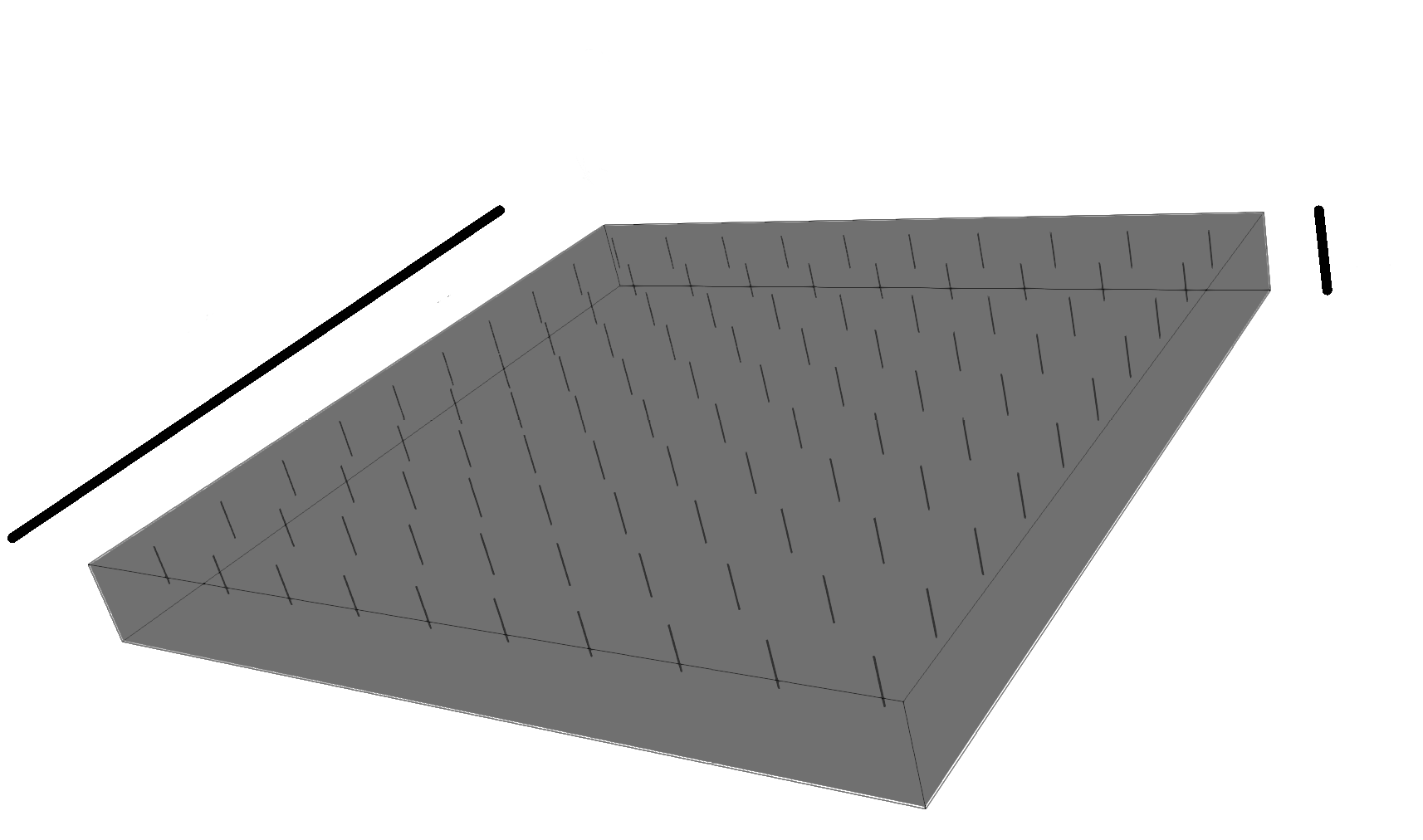}
\put(214,95){$M$}
\put(32,80){$1$}
\put(160,108){$\Omega^\ve$}
\linethickness{3pt}
\put(157,112){\vector(-1,-2){6}}
\put(143,21){\vector(1,-4){5}}
\put(143,21){\vector(3,4){12}}
\put(143,21){\vector(-4,1){20}}
\put(157,30){$x_1$}
\put(120,17){$x_2$}
\put(151,0){$x_3$}
\color{black}
\end{overpic} \label{full_domain}}
 \subfloat[Periodicity~cell]{\begin{overpic}[width=0.35\textwidth,tics=9]{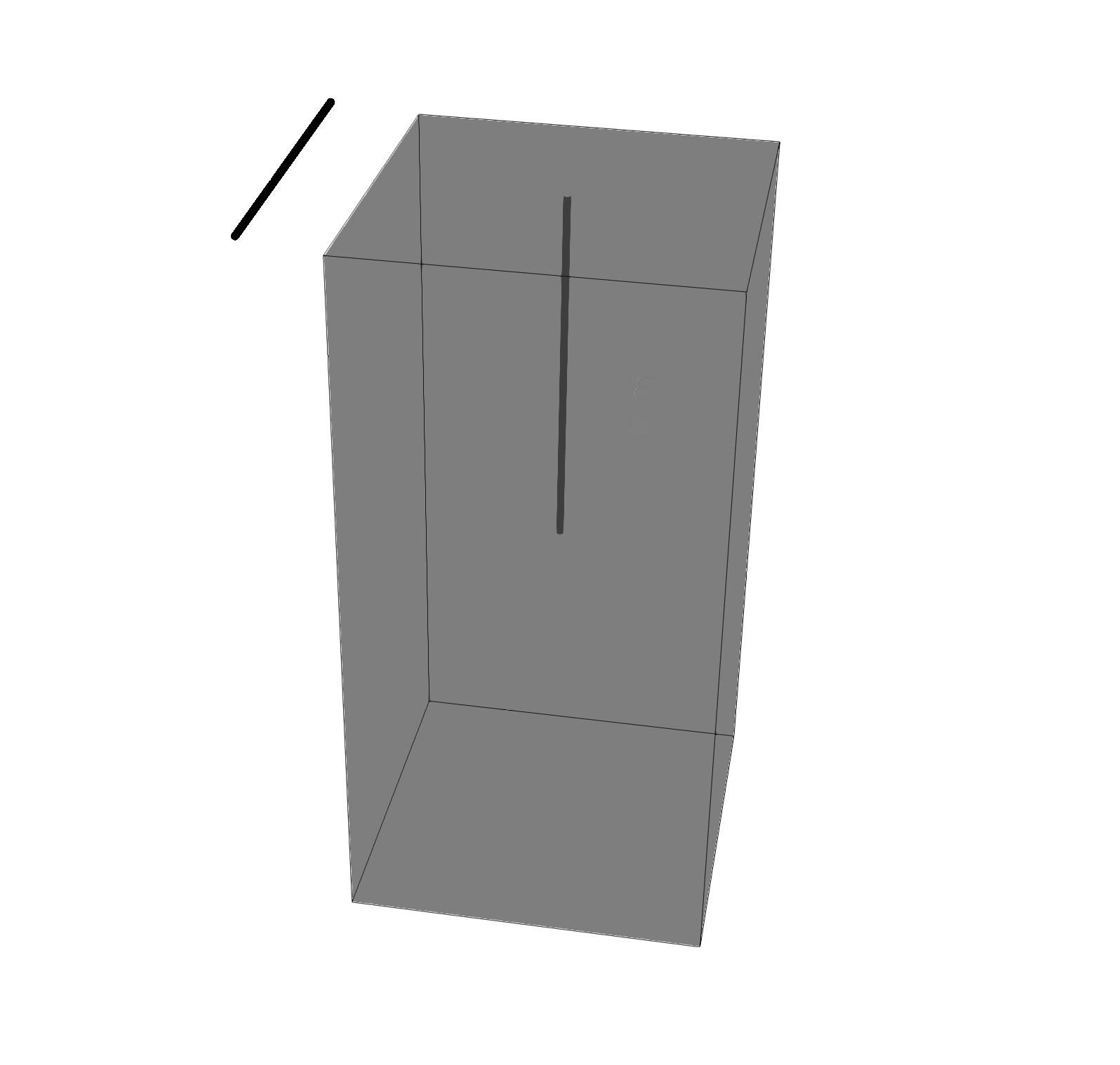}
\put(28,108){$\ve$}
\put(73,75){$\Gamma^\ve$}
\end{overpic} \label{unit_cell}}
\caption{Problem geometry}
\end{figure}
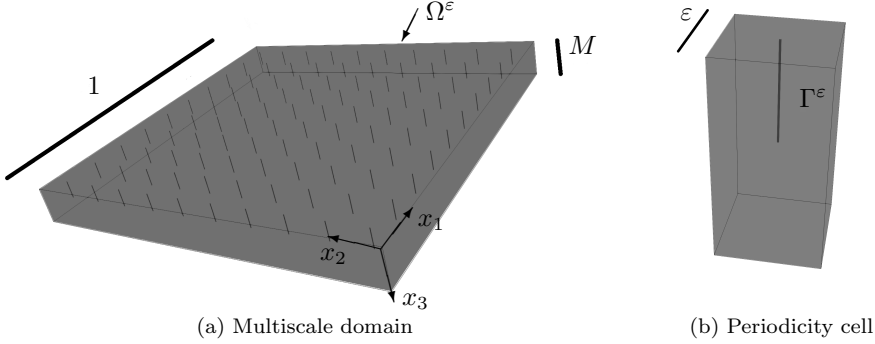

Consider a domain  $\Omega = G  \times (0,M)$ around a single plant root, with $M>0$ being representative of the half-distance between neighboring roots, where the Lipschitz domain $G \subset  \mathbb R^2$ represents the part of the root surface under consideration. 
We assume that the root hairs are circular cylinders (of dimensionless length $L$, with $L<M$,  and radius $r_\ve$) orthogonal to the (planar) root surface, on which they are periodically distributed, see Figure~\ref{full_domain}.  A single root hair can be described~as
$$  B_{r_\ve} \times (0,L), \quad \text{ where } \;  B_{r_\ve} = \{ (x_1, x_2) \in \mathbb R^2 \; : \; x_1^2 + x_2^2 < r_\ve^2 \}. $$  
%
Denoting by $Y=(-1/2, 1/2)^2$ the unit cell, and taking $\ve$ to be the small parameter (the representative distance between the root hairs being small compared to the root length), the set of root hairs belonging to the root surface 
can be written as
$$
\Omega^\ve_{1,L} =\bigcup_{\xi\in \Xi^\ve}  (\overline B_{r_\ve} +\ve \xi)\times (0,L),
 \; \;  \text{ with } \;  \Xi^\ve = \{ \xi \in \mathbb Z^2 \; : \;  \ve  (Y + \xi)\  \subset G \},
$$
i.e. we only include the root hairs whose base is fully contained in $G$. The solution domain is then defined by  $\Omega^\ve = \Omega \setminus \Omega_{1,L}^\ve$.

We assume the root hairs to be sparsely distributed, i.e.\  $r_\ve \ll \ve \ll 1$, define $a_\ve=  r_\ve/ \ve \ll 1$, and assume that $M = O(1)$ and $L = O(1)$.  The surfaces of the root hairs are given by
$$
\Gamma^\ve = \bigcup_{\xi\in \Xi^\ve}  (\partial B_{r_\ve} +\ve \xi)\times (0,L) .
$$
We shall also use the notation
$ \Omega_L = G \times (0,L) $
corresponding to the range of $x_3$ occupied by root hairs. 

Outside the root hairs we consider  the diffusion of  nutrients
\begin{equation}\label{main_1}
\begin{aligned}
\partial_t u_\ve = \nabla\cdot( D_u \nabla  u_\ve)  \qquad \text{ in } \quad \Omega^\ve, \; \;   t >0,
\end{aligned}
\end{equation}
with constant (dimensionless) diffusion coefficient  $D_u>0$, and assume that  nutrients are taken up on the root surface according to
\begin{equation}\label{bc_root}
\begin{aligned}
D_u \nabla u_\ve \cdot {\bf n } = - \beta\,  u_\ve \qquad  \text{on }  \quad \Gamma_R^\ve, \; \;  t>0,
\end{aligned}
\end{equation}
where $\Gamma_R^\ve =\overline{ \Omega^\ve} \cap  \{x_3 = 0 \}$  defines the surface of the root (excluding the root hairs)\footnote{Even though the analysis for a nonlinear boundary condition would be straightforward, we consider linear uptake here, as the emphasis will be on the derivation of sink terms resulting from the boundary conditions applied on the hair surfaces, which often are dominant in nutrient uptake.}, and on the surfaces of the root hairs
\begin{equation}\label{bc_hair}
\begin{aligned}
D_u \nabla u_\ve \cdot  {\bf n } = -\ve  K(a_\ve)  \,  g(u_\ve) \qquad  \text{on }  \quad \Gamma^\ve,  \; \;  t>0,
\end{aligned}
\end{equation}
where ${\bf n}$ denotes the outer-pointing  unit normal vector to $\partial \Omega^\ve$, $\beta \geq 0$ is an uptake rate, $g(\eta)$ is  smooth  (continuously differentiable) and monotone non-decreasing for $\eta \in [ - \tilde \varsigma, \infty)$, with  some $\tilde \varsigma >0$,    and $g(\eta) = g_1(\eta) + g_2(\eta)$, where  $g_1(\eta) \geq 0$ for $\eta\geq 0$, with $g_1(0) = 0$, and $g_2$ is sublinear, with $g_2(0) \leq 0$. The monotonicity of  $g$  ensures existence of a unique solution  $h$ of  $h+ \sigma g(h)= \zeta$,  with  $\zeta \ge 0$ and $\sigma >0$, important for the derivation of   macroscopic equations for   \eqref{main_1}-\eqref{bc_hair},~\eqref{bc_rest},~\eqref{ic}. In  Section \ref{sec_numerics} we will consider the Michaelis-Menten type function
\begin{equation}\label{eqn:MM_form}
g(u) = \frac{u}{1+u},
\end{equation}
often used in modelling uptake processes by plant roots, e.g.\  \cite{claassen1974, epstein1952}, for which all of the above assumptions are satisfied, with $g_2 \equiv 0$.
The scaling factor $K(a_\ve)$ in~\eqref{bc_hair}~is~set~to~be 
\begin{equation} \label{scaling_K}
K(a_\ve) = \frac \kappa{a_\ve}, 
\end{equation}
with some positive constant $\kappa =O(1)$ (see the Supplementary materials for the justification of this scaling). On other parts of the boundary  $\partial \Omega^\ve$  we consider 
\begin{equation}\label{bc_rest}
\begin{aligned}
D_u \nabla u_\ve \cdot {\bf n } = 0 \qquad  \text{ on }  \quad \partial \Omega^\ve  \setminus (\Gamma^\ve\cup \Gamma_R^\ve), \; \;  t>0.
\end{aligned}
\end{equation}
The initial nutrient concentration is given by
\begin{equation}\label{ic}
u_\ve(0,x) = u_{\rm in}(x)  \quad \text{ for } x\in \Omega^\ve,
\end{equation}
where we assume that $u_{\rm in} \in H^2(\Omega)$ and $0 \leq u_{\rm in}(x) \leq u_{\rm max}$ for $x\in \Omega$.

First we consider the definition of a weak solution of   \eqref{main_1}--\eqref{bc_hair}, \eqref{bc_rest}, and \eqref{ic}.  We shall use the notations  $\Omega_T^\ve = (0,T)\times \Omega^\ve$, $\Gamma_T^\ve = (0,T)\times \Gamma^\ve$, and $\Gamma_{R, T}^\ve = (0,T)\times \Gamma_R^\ve$.
\begin{definition}
A weak solution of  problem \eqref{main_1}--\eqref{bc_hair}, \eqref{bc_rest},  \eqref{ic} is a function $u_\ve \in L^2(0,T;  H^1(\Omega^\ve))$, with $\partial_t u_\ve \in L^2((0, T) \times \Omega^\ve)$,  satisfying
\begin{equation}\label{main_weak}
\begin{aligned}
\int_{\Omega^\ve_T} \hspace{-0.1 cm } \big( \partial_t u_\ve \phi  + D_u \nabla u_\ve \cdot \nabla \phi \big) dx dt = -  \ve    \int_{\Gamma^\ve_T}\hspace{-0.1 cm}  \frac \kappa{a_\ve} g(u_\ve) \phi \,  d\gamma^\ve dt -   \int_{\Gamma_{R, T}^\ve} \hspace{-0.2 cm }  \beta\,  u_\ve \phi  \, d\gamma^\ve dt
\end{aligned}
\end{equation}
for $\phi \in L^2(0,T;  H^1(\Omega^\ve))$ and $u_\ve(t) \to u_{\rm in}$  in $L^2(\Omega^\ve)$ as $t\to 0$.
\end{definition}
Standard results for  parabolic equations, together with the above assumptions on $g$, ensure the existence of a unique  weak solution of problem   \eqref{main_1}--\eqref{bc_hair}, \eqref{bc_rest},  \eqref{ic}  for any fixed $\ve >0$, see e.g.\  \cite{ladyzh, lieberman}.

\section{Derivation of the macroscopic equations using the method of formal asymptotic expansions}\label{formal_derivation}
To derive the macroscopic equations from the multiscale problem \eqref{main_1}--\eqref{bc_hair}, \eqref{bc_rest},  \eqref{ic} we  first apply the method of the formal asymptotic expansions.  We shall consider different scalings for $a_\ve$ and derive equations for zero, first and second orders of approximation for solutions.  
 Apart from the macroscopic variables $x = (x_{1},x_{2},x_{3})$, 
 we further introduce $y= (y_{1},y_{2}) = (x_{1}/ \ve, x_{2}/\ve)$ and $z= ( z_{1},z_{2} ) = (x_{1}/ r_\ve, x_{2}/r_\ve) =  (y_{1}/ a_\ve, y_{2}/a_\ve)$. 
 Since there is no microscopic variation in the $x_3$ direction, we do not include any dependence on $y_3$ (or $z_3$). Notice that due to the assumed scale separation between the radius of the root hairs and the distance between them, three scales are present: an inner microscopic scale, $\|z\| = \sqrt{ z_1^2 + z_2^2} = O(1)$, corresponding to the radius of root hairs, an outer microscopic scale, $\|y\| = O(1)$, given by the distance between them and a macroscopic scale, $\|x\| = O(1)$, corresponding to a representative length of a plant root (for simplicity, we assume that the typical distance between two neighboring roots is of the same order as the representative root length).

 In the derivation of macroscopic equations we consider two cases.  In the first, we take the limits in the order $\ve \to 0$ then $a_\ve \to 0$, with no relationship assumed between these two parameters and, in the second, we study a distinguished limit motivated by the analysis in the first section. Note that in the first case, instead of $a_\ve$, we suppress the subscript to recall that $a$ and $\ve$ are independent small parameters therein.
 
  \subsection{Derivation of the macroscopic equations in the case of complete 
  scale separation between $\ve$ and $a$} \label{subsec_diffScaling}
In this section, we assume complete scale separation between $\ve$ and $a$ (i.e.\  we take the limit $\ve \to 0$ followed by $a \to 0$). We adopt the ansatz
\begin{equation}\label{ansatz}
\begin{aligned}
u_{\ve}(t, x, a)=u_{0}(t,x, \hat x/\ve, a) + \ve u_{1}(t,x, \hat x /\ve, a) + \ve^{2} u_{2}(t,x, \hat x /\ve, a) + \cdots,
\end{aligned}
\end{equation}
 for $x\in \Omega_L$,  $t>0$, $\hat x = (x_1, x_2)$, and  $u_j (t,x, \cdot, a)$  being $Y$-periodic (cf. \cite{Bensoussan_1978, kevorkian2012}). We first fix $0<a<1/2$, then perform a separate $a \to 0$ analysis at each order in $\ve$. Note that for the simplicity of presentation, we will consider linear boundary condition in \eqref{bc_hair}, i.e.\ $g(u) = u$; the same calculations have also been performed for a nonlinear  function $g(u)$ by Taylor expanding of $g(u)$ about $u_0$ (see the Supplementary materials).
 
\subsubsection{$a=O(1)$}
Even though this problem has already been analyzed 
in \cite{leitner2010_1, Ptashnyk2010b},  to set up for the sublimit $a \to 0$ in the next section,
we briefly recall the main outcomes of this analysis. The terms of order  $\ve^{-2}$ in \eqref{main_1} and of order $\ve^{-1}$ in \eqref{bc_hair} yield
\begin{equation} \label{u0_11}
\nabla_{y}\cdot (D_u \nabla_y u_{0}) =0  \quad  \text{ in } Y_{a}, \quad D_u \nabla_{y} u_{0} \cdot  \hat{\bf n } = 0 \quad  \text{ on } \Gamma_{a}, \quad u_0 \text{ is }  \; Y\text{-periodic},
\end{equation}
where $Y_{a} = Y\setminus \overline B_{a}$,  $\Gamma_{a} = \partial B_{a}$. The existence and uniqueness theory for linear elliptic equations with zero-flux and periodic boundary conditions  implies that solutions of~\eqref{u0_11}  are independent of $y$, i.e.
$ u_0 = u_0( t,x, a)$.
 For the terms of order  $\ve^{-1}$ in  \eqref{main_1} and of order $\ve^{0}$ in \eqref{bc_hair} we then have
\begin{equation}\label{eq_u1}
 \nabla_{y}\cdot(D_u \nabla_y u_{1}) =0  \quad  \text{ in } Y_{a}, \qquad
  D_u\nabla_{y} u_{1} \cdot \hat{\bf n } = - D_u \nabla_{\hat x} u_{0} \cdot \hat{\bf n }  \; \; \; \; \text{ on }  \Gamma_{a},
 \end{equation}
and $u_1$ is $Y$-periodic, where $\hat x =(x_1, x_2)$. The solution reads
\begin{equation}\label{ansatz_u_1_form}
 u_{1} (t,x, y, a)= U_1(t,x, a) + \nabla_{\hat x} u_{0}(t,x,a) \cdot \boldsymbol{\nu}(y, a),
\end{equation}
where $U_1$ consists of contributions to $u_1$ that do not depend on the microscale and the  vector function $\boldsymbol{\nu}(y,a) = (\nu_{1}(y,a) , \nu_{2}(y,a))$     is a solution of 
\begin{equation}\label{eq_chi_1}
\begin{aligned}
\nabla_{y}\cdot(D_u \nabla_y \boldsymbol{\nu} )=0  \hspace{0.2cm} \text{ in } Y_a, \qquad
\nabla_{y}\boldsymbol{\nu} \cdot \hat{\bf n } = - \hat{\bf n } \hspace{0.2cm} \text{ on }  \Gamma_a, \qquad \boldsymbol{\nu} \; \text{ is }  \; Y\text{-periodic}.
\end{aligned}
\end{equation}
Finally, collecting  the terms of order  $\ve^0$  in \eqref{main_1} and of order  $\ve$ in  \eqref{bc_hair} yields
\begin{eqnarray}\label{eq_u2}
 \nabla_{y}\cdot(D_u \nabla_y u_2) && =   \partial_t u_{0}  - \nabla_{x}\cdot(D_u \nabla_x u_{0})    - \nabla_{\hat x}\cdot(D_u \nabla_{y} u_{1}) - \nabla_y\cdot (D_u \nabla_{\hat x} u_1)   \; \;  \text{ in } Y_{a},   \nonumber \\
&& D_u \nabla_y u_2 \cdot  \hat{\bf n } =  -K(a) u_{0} -  D_u \nabla_{\hat x} u_1 \cdot  \hat{\bf n }   \hspace{3.3 cm}     \text{ on  }   \Gamma_{a}.
\end{eqnarray}
Integrating \eqref{eq_u2} over $Y_a$ and using the divergence theorem (for more details see  \cite{kory2018}) gives as the leading-order macroscale problem
\begin{equation}\label{homo_fixed_a_delta}
\partial_t u_0 = \nabla_{x} \cdot \left( D_u \boldsymbol{D}_{\mathrm{eff}} (a) \nabla_x u_0 \right) - \frac{2 \pi a K(a)}{1- \pi a^2} u_0,
\end{equation}
where $ \boldsymbol{D}_{\mathrm{eff}}(a) = \boldsymbol{I} + \boldsymbol{B}(a)/(1- \pi a^2)$, $\boldsymbol{I}$ is the identity matrix and
\begin{equation}
\label{eqn:B_matrix}
\boldsymbol{B}(a) =
\begin{pmatrix}
\int_{Y_a} \frac{\partial \nu_{1}(y,a)}{\partial y_{1}} dy  & 0 & 0  \\
0 & \int_{Y_a} \frac{\partial \nu_{2}(y,a)}{\partial y_{2}} dy & 0  \\
0 & 0 & 0
\end{pmatrix}.
\end{equation}
\subsubsection{$a \ll 1$}\label{sec:independent_a_small}
Now, we analyze  \eqref{eq_chi_1} and \eqref{homo_fixed_a_delta} in the  limit $a \to 0$. Because of the large scale difference between the periodicity of the microscopic structure and the radius of the root hairs, in the analysis of the asymptotic behavior of the solution we can distinguish between the behavior in a region characterized by $\|z\| = O(1)$, which will correspond to an inner solution (denoted using a superscript $^I$) and the behavior in a region characterized by $\|y\| = O(1)$, corresponding to an outer solution (denoted using a superscript $^O$),  see~\cite{kory2018} for more details. Thus each term in \eqref{ansatz} requires its inner and outer analysis, some of which will involve expanding in $\delta = 1/ \ln(a^{-1}) \ll 1$. These logarithmic relationships arise due to the two-dimensional microstructure, reflecting the fact that the Green function of the Laplace operator in $\mathbb{R}^2$ is proportional to $\ln(r)$, as will become obvious at $O(\ve^2)$. Note that for any $n \geq 2$, we have
\begin{equation}\nonumber
\cdots \ll \ve^n \ll \cdots \ll \ve \ll \cdots \ll a^n \ll \cdots \ll a \ll \cdots \ll \delta^n \ll \cdots \ll \delta = 1/ \ln(a^{-1}) \ll 1,
\end{equation}
due to the assumption of the complete scale separation between $a$ and $\ve$. We expand
\begin{equation}\label{ansatz_a_I}
u_0(t, x, \delta) = u_{0,0}(t,x) + o(1).
\end{equation}
The macroscopic behaviour of $u_{0,0}$ will be determined via Fredholm alternative at $O(\ve^2)$ (see \eqref{eqn:homo_small_a_just_CE}). Proceeding to $O(\ve)$, we should not aim to satisfy the boundary condition from \eqref{eq_chi_1} on $\Gamma_{a}$ in the $\|y\| = O(1)$ region (this part of the boundary degenerates to a point in the limit $a \to 0$) and we have an expansion
\begin{equation}\label{chi_1_outer}
\boldsymbol{\nu}^O(y,a) =\boldsymbol{\nu}_{0}^O(y) + a\, \boldsymbol{\nu}_{1}^O(y) + \cdots,
\end{equation}
with $\boldsymbol{\nu}_{i}^O$ being $Y$-periodic  and satisfying Laplace's equation. Setting $z= y/a$ in \eqref{eq_chi_1} yields 
\begin{equation}\label{eq_w0}
\begin{aligned}
&\nabla_{z}\cdot(D_u \nabla_z \boldsymbol{\nu} )=0  \quad && \text{ in } Y_{1/a}, \qquad
\nabla_{z}\boldsymbol{\nu} \cdot \hat{\bf n } = - a \hat{\bf n } \quad && \text{ on }  \partial B_1,
\end{aligned}
\end{equation}
where $Y_{1/a} = a^{-1}  Y \setminus \overline B_1$. This suggests an inner expansion of the form
\begin{equation}\label{chi_1_inner}
\boldsymbol{\nu}^I(z,a) =\boldsymbol{\nu}_{0}^I(z) + a\, \boldsymbol{\nu}_{1}^I(z) + \cdots  . 
\end{equation}
It follows that $\boldsymbol{\nu}_{0}^I$ is independent of $z$ and 
\begin{equation}
\begin{aligned}
&\boldsymbol{\nu}_{1}^I(z) = -  \Big[\alpha\Big(r + \frac 1r\Big) + r\Big] \frac{(z_1, z_2)}{r} ,
\end{aligned}
\end{equation}
where $r=\|z\|$, and $\alpha = -1$ is required to match with the outer region. 
Hence
\begin{equation}\label{chi_11_I}
\begin{aligned}
\boldsymbol{\nu}_{1}^I(z) = \frac{(z_1, z_2)}{\|z\|^2} .
\end{aligned}
\end{equation}
To match the inner  $\boldsymbol{\nu}^I$ and outer $\boldsymbol{\nu}^O$, \eqref{chi_1_outer} has to contain terms of the form
$$   a \frac{ (z_1 , z_2)} {\|z\|^2} =  a^2 \frac{(y_1, y_2)}{\|y\|^2} $$
as $\|y\| \rightarrow 0$. Noting that the solution of 
$$\Delta_y \boldsymbol{v}(y) = 2 \pi \nabla_{y} \delta(y), \quad \boldsymbol{v} \; \text{ is }  \; Y\text{-periodic},$$
where $\delta(y)$ is the Dirac delta, 
 has the behavior
$$ \boldsymbol{v}(y) \sim \frac{( y_1 , y_2 )^T}{\|y\|^2}  \quad \text{ as }   \; \|y\| \rightarrow 0, $$
we infer that $\boldsymbol{\nu}_{2}^O = \boldsymbol{v}$. 
In order to uncover the effective behavior at the macroscale, we need to analyze \eqref{eq_u2} in the inner and outer regions and matching between these will eventually lead us to the homogenized equation~\eqref{eqn:homo_small_a_just_CE}. Using the information on the inner and outer behavior of $u_1$, see \eqref{ansatz_u_1_form} and \eqref{chi_11_I},  problem \eqref{eq_u2} becomes
\begin{eqnarray}\label{eq_u2_using_info}
\begin{aligned}
 \nabla_{y}\cdot(D_u \nabla_y u_2)  & = \hspace{0.1cm} \partial_t u_{0}  - \nabla_{x}\cdot(D_u \nabla_x u_{0}) +O(a)   \; \; &&  \text{ in } Y_{a},  \\
D_u \nabla_y u_2 \cdot  \hat{\bf n }  & = - K(a) u_{0}  -  D_u \nabla_{\hat x} \left( U_1 +  \nabla_{\hat x} u_{0} \cdot \boldsymbol{\nu} \right) \cdot  \hat{\bf n } \hspace{0.5 cm}    && \text{ on  }   \Gamma_{a} .
\end{aligned}
\end{eqnarray}
Rescaling by $z= y/a$ and using~\eqref{scaling_K}, we obtain
\begin{eqnarray}\label{eq_u2_using_info_rescaled}
\begin{aligned} 
\nabla_{z}\cdot(D_u \nabla_z u_2)
= & \hspace{0.1cm} O(a^2) && \text{ in }  Y_{1/a}  , \nonumber \\
  D_u \nabla_{z} u_2 \cdot \hat{\bf n } = & - \kappa \, u_0+O(a) && \text{ on }  \partial B_1, \nonumber
  \end{aligned} 
\end{eqnarray}
Recalling \eqref{ansatz_a_I}, we infer the following ansatz for $u_2$ 
\begin{equation}\label{u_2^I_form}
u_2(t,x,y,\delta) = U_2(t,x, \delta) + u_0(t,x, \delta) \psi(y,\delta),
\end{equation}
where the inner ($z= y/a = O(1)$) expansion for $\psi$ reads
\begin{equation}\label{chi_2_I_ansatz}
\psi^I(z, \delta) = \psi_{0}^I(z) + O(\delta)
\end{equation}
and at the leading order we get
\begin{equation}\label{eq_chi2_I}
\begin{aligned}
&\nabla_{z}\cdot(D_u \nabla_z \psi_{0}^I )=0  \quad && \text{ in } Y_{\infty}, \qquad
D_u \nabla_{z} \psi_{0}^I \cdot \hat{\bf n } = - \kappa \quad && \text{ on }  \partial B_1,
\end{aligned}
\end{equation}
where $Y_{\infty} = \mathbb{R}^2 \setminus \overline B_1$, the solution of which reads
\begin{equation}
\psi_{0}^I(z) = (\kappa/D_u) \ln{(\|z\|)}.
\end{equation}
Rewriting this in the outer variables $y$, we obtain
\begin{equation}\label{chi_20^I_in_y}
(\kappa/ D_u) \big( \ln{(\|y\|)} + \delta^{-1} \big).
\end{equation}
In the $\|y\| = O(1)$ region, the ansatz \eqref{u_2^I_form} (rescaled to $y$ variables) together with \eqref{chi_20^I_in_y} results in an outer expansion for $\psi$ of the form
\begin{equation}\label{chi_2_O_ansatz}
\psi^O(y, \delta) = \psi_{-1}^O(y) \delta^{-1} + \psi_{0}^O(y) + O(\delta),
\end{equation}
which means that the substitution of \eqref{u_2^I_form} into \eqref{eq_u2_using_info} gives at the leading order
\begin{equation}
\nabla_{y}\cdot(D_u \nabla_y \psi_{-1}^O) = 0   \; \;  \text{ in } Y, \quad \psi_{-1}^O \; \text{ is }  \; Y\text{-periodic}
\end{equation}
implying that $\psi_{-1}^O$ is independent of $y$. At the next order in the outer expansion, we need to capture the logarithmic contribution from \eqref{chi_20^I_in_y} (required for matching with the inner solution), and we thus conclude
\begin{equation}\nonumber
\begin{aligned}
&u_{0,0} \nabla_{y}\cdot(D_u \nabla_y \psi_{0}^O) = \hspace{0.1cm}  \partial_t u_{0,0}  - \nabla_{x}\cdot(D_u \nabla_x u_{0,0}) - 2 \pi \kappa u_{0,0} \hspace{0.03cm} \delta(y) \;  \; \text{ in } Y, \\
&\psi_0^O \qquad \text{ is }  Y\text{-periodic}.    
\end{aligned} 
\end{equation}
Due to the Fredholm alternative this problem admits a solution if and only if  
\begin{equation}\label{eqn:homo_small_a_just_CE}
\partial_t u_{0,0} =  \nabla_{x} \cdot(D_u \nabla_{x} u_{0,0}) - 2 \pi  \kappa\,  u_{0,0}  \quad \text{ for } \; x\in \Omega_L, \; t>0. 
\end{equation}
We have thus obtained an outer approximation 
\begin{eqnarray}\label{final_outer_separation}\nonumber 
u_\ve = \Big[ u_{0,0}(t,x) + \cdots \Big] + \ve \Big[  U_{1,0}(t,x) +\boldsymbol{\nu}_{0}^O(y) \cdot \nabla_{\hat{x}} u_{0,0}(t,x) + \cdots  \Big]  \\
+ \ve^2 \Big[ U_{2,0}(t,x) + \delta^{-1} u_{0,0}(t,x) \psi_{-1}^O(y) + \cdots \Big] + \cdots.
\end{eqnarray}
Note as a consistency check that we could have also arrived at \eqref{eqn:homo_small_a_just_CE} more directly via the $a \to 0$ limit in \eqref{homo_fixed_a_delta} (for details, see section 4.2 in \cite{kory2018}). However, in general, as we have $\delta^{-1} \gg 1$, the $\ve^2 \delta^{-1}$ term could be promoted to $O(\ve)$ or even $O(1)$, depending on the specified limit behavior of $\delta$ with respect to $\ve \to 0$, thereby identifying the distinguished limit that we consider below.


\subsection{Derivation of macroscopic equations: distinguished limit}\label{sec:dist_limits}
In the asymptotic analysis in Section~\ref{subsec_diffScaling} we first took the limit $\ve \to 0$, and then $a_\ve \to 0$.
Motivated by the $\ve^2 \delta^{-1}$ term (with $\delta^{-1} = \ln(1/a_\ve)$) from \eqref{final_outer_separation}, in this section we consider the situation where $\ve$ and   $\ln(1/a_\ve)$ are dependent and analyze two cases, $\ve \ln(1/a_\ve) = O(1)$ (section \ref{sec:sparse_first_nonuni}) and $\ve^2 \ln(1/a_\ve) = O(1)$ (section \ref{sec:sparse_next_nonuni}).  
Note that even though the  case  $\ve \ln(1/a_\ve) = O(1)$ does not give us a distinguished limit, the $O(\ve)$ balance changes and thus this case is still worth studying. In both cases we set $K(a_\ve) = \kappa/a_\ve$ and use the formal asymptotic expansion
\begin{equation} \label{eqn:ansatz_dist}
u(t, x , \ve) = u_{0}(t, x, {\hat x}/{\ve}) + \ve u_{1}(t, x, {\hat x}/{\ve}) + \ve^{2} u_{2}(t, x, {\hat x}/{\ve}) +  \ve^{3} u_{3}(t, x, {\hat x}/{\ve}) + \cdots
\end{equation}
to derive the macroscopic equations, $u_{j}$ being $Y$-periodic with respect to the outer microscopic variables $y=\hat x/\ve$. 
The convergence of solutions of the multiscale problems to solutions of the derived macroscopic equations 
will subsequently be confirmed via rigorous analysis in Section~\ref{convergence} and  numerical simulations in Section~\ref{sec_numerics}.

We consider a linear function $g(u)= u$ in the boundary condition \eqref{bc_hair}, the details on derivation of the macroscopic equations for nonlinear boundary conditions are given in the Supplementary materials. In  the next two subsections, $\lambda$ is an $O(1)$ quantity, with a different meaning in each subsection.

\subsubsection{Derivation of macroscopic equations in the case $\ve \ln(1/a_\ve) = \lambda$} \label{sec:sparse_first_nonuni}
Observe first that the $\ve^2 \delta^{-1}$ term from \eqref{final_outer_separation} becomes $O(\ve)$ here and therefore we do not expect it to impact on the leading order. The ansatz  \eqref{eqn:ansatz_dist} yields
\begin{equation}\label{outer_22}
\begin{aligned}
&\partial_t (u_{0} +  \ve u_1 +  \cdots)  =\Big(\frac 1{\ve^2}  \cA_0 + \frac 1 \ve \cA_1 + \cA_2\Big)( u_{0} + \ve u_{1} +  \cdots)  && \text{in } \Omega_L \times Y_{a_\ve},   \hspace{-0.5cm} \\
&D_u \Big(\frac 1 \ve  \nabla_{y} +  \nabla_{\hat x} \Big) \left( u_{0} + \ve u_{1} +  \cdots \right) \cdot \hat{\bf n } = - \kappa\,  e^{\frac{\lambda}{\ve}} \ve  ( u_{0} + \ve u_1 + \cdots) && \text{on } \Omega_L \times \Gamma_{a_\ve},   \hspace{-0.5cm}
 \end{aligned}
 \end{equation}
 where 
 \begin{equation}
 \begin{aligned} \nonumber
 \cA_0 v  \equiv \nabla_y \cdot (D_u \nabla_y v), \hspace{0.15cm} \cA_1 v  \equiv \nabla_y \cdot (D_u \nabla_{\hat x} v)+ \nabla_{\hat x} \cdot (D_u \nabla_y v), \hspace{0.15cm} \cA_2 v  \equiv \nabla_{x} \cdot (D_u \nabla_x v).
 \end{aligned}
 \end{equation}
 On the root surface we have
 $$
 D_u \Big(\frac 1 \ve  \nabla_{y} +  \nabla_x \Big) ( u_{0} + \ve u_{1} + \ve^{2} u_{2}  + \cdots ) \cdot {\bf n } = - \beta \left( u_{0} + \ve u_1 + \cdots \right) \; \;  \text{ on } \big\{ x_3 = 0 \big\} \times Y_{a_\ve}.
 $$
 As in Section~\ref{subsec_diffScaling} we analyze the behavior of solutions for $\|z\| = O(1)$ and $\|y\| = O(1)$ successively. The scaling $z= y/{a_\ve} = y\,  e^{\lambda/\ve}$ implies
\begin{equation}\label{inner_22}
\begin{aligned}
\partial_t u_0 + \ve \partial_t u_1 + \cdots =  \Big( \frac{e^{\frac{2 \lambda}{\ve}}} {\ve^2} \cB_0 +    \frac{ e^{\frac{ \lambda}{\ve}} }{\ve} \cB_1  +  \cA_2\Big) ( u_{0} + \ve u_{1}  +  \cdots )  &\;  \text{ in } \Omega_L \times Y_{1/a_\ve},  \hspace{-0.5 cm } \\
D_u\Big( \frac{e^{\frac \lambda \ve}}\ve  \nabla_{z} +  \nabla_{\hat x} \Big)( u_{0} + \ve u_{1} +  \cdots) \cdot \hat{\bf n } = - \kappa\,  \ve   e^{\frac{\lambda}{\ve}} (u_{0} + \ve u_1 + \cdots ) &\;  \text{ on } \Omega_L \times \partial B_1,  \hspace{-0.5 cm }
\end{aligned}
\end{equation}
where
\begin{equation}\label{mathcal_B}
 \cB_0 v  \equiv \nabla_z \cdot (D_u \nabla_z v), \quad \cB_1 v  \equiv \nabla_z \cdot (D_u \nabla_{\hat x} v)+ \nabla_{\hat x} \cdot (D_u \nabla_z v).
 \end{equation}
The inner approximations satisfy
\begin{equation}\label{eq_inner_22}
\begin{aligned}
& \nabla_{z}\cdot(D_u \nabla_z u_{j}^{I}) =0\;   \text{ in }   Y_{\infty} ,    \; && D_u \nabla_{z} u_{j}^{I} \cdot \hat{\bf n } =0 && \text{on }   \partial B_1 ,  \; j=0,1, \\
 &\nabla_{z}\cdot(D_u \nabla_z u_{j}^{I}) =0 \;  \text{ in }   Y_{\infty} ,  \;   && D_u   \nabla_{z} u_{j}^{I} \cdot \hat{\bf n } = - \kappa\,  u_{j-2}^{I} && \text{on }   \partial B_1,\;  j=2,3,4,
 \end{aligned}
  \end{equation}
which imply
\begin{equation}
\begin{aligned}
& u_{0}^{I} (t,x,z)= u_{0}^{I}(t, x), \quad u_{1}^{I} (t,x,z)= u_{1}^{I}(t, x), 
\\
& u_{j}^{I}(t,x,z) = \frac \kappa{D_u} u_{j-2}^{I}(t,x) \ln{(\|z\|)} + U^I_j(t,x), \qquad \text{ for } \; \; j=2,3, \\
& u_{4}^{I}(t,x,z) = \frac \kappa{D_u} U^{I}_2(t,x) \ln{(\|z\|)} + U_4^I(t,x).
\end{aligned}
\end{equation}
Note that in this section we expand up to $O(\ve^4)$, because we wish to find a two-scale approximation valid up to $O(\ve^2)$ and compare it with full-geometry numerical simulation results in Section~\ref{sec_numerics}.
The outer approximations satisfy
\begin{equation}
 \nabla_{y}\cdot(D_u \nabla_y  u_{0}^O) = 0 \quad  \text{ in } Y, \; \; \;  \; \; \; \; u_{0}^O \; \;  \; \; \; Y-\text{periodic}, \label{eq_outer_220}
\end{equation}
so $u_{0}^{O}(t,x,y) = u_{0}^{O}(t, x)$ and therefore $u_{1}^{O}(t,x,y) = u_{1}^{O}(t, x)$ holds similarly. 
Since in the outer microscopic variables we have 
$$
\begin{aligned}
u_{2}^{I}(t,x,z)  
&\,  = \frac \kappa{D_u} \Big[ u_{0}^{I}(t,x) \ln{(\|y\|)} +  u_{0}^{I}(t,x) \frac{\lambda}{\ve}\Big] +  U^I_2(t,x),
 \end{aligned}
 $$
to match logarithmic terms in outer and inner approximations we consider 
\begin{equation} \label{eqn:first_nonuni}
\begin{aligned}
\nabla_{y}\cdot (D_u \nabla_y u_{2}^O) =  \partial_t u_{0}^O - \nabla_{x}\cdot(D_u \nabla_x u_{0}^O) + 2 \pi \kappa u_{0}^I\,\delta (y) && \text{ in }  \; Y
\end{aligned}
\end{equation}
and $u_{2}^O$ is $Y$-periodic. The solvability condition for \eqref{eqn:first_nonuni} yields
\begin{equation} \label{eqn:first_nonuni_OI}
\partial_t  u_{0}^O = \nabla_{x}\cdot(D_u \nabla_x u_{0}^O)- 2 \pi  \kappa u_{0}^I  \quad \text{ for } \; x\in \Omega_L, \; t>0,
\end{equation}
and substituting this result  into \eqref{eqn:first_nonuni} gives
\begin{equation} \label{eqn:first_nonuni_2}
\begin{aligned}
\nabla_{y}\cdot (D_u \nabla_y u_{2}^O) =  2 \pi  \kappa  \, \big( \delta (y) - 1 \big) u_{0}^I\, \quad \text{ in } \; Y.
\end{aligned}
\end{equation}
Therefore
\begin{equation} \label{ansatz_u2}
\begin{aligned}
u_{2}^O(t,x,y) =  U_2^O(t,x) + 2 \pi (\kappa/D_u) u_{0}^I(t,x) \psi(y) \quad \text{ for } \; x\in \Omega_L, \; t >0,
\end{aligned}
\end{equation}
where $\psi(y)$ is a solution (unique up to a constant) of
\begin{equation} \label{cellproblem_ve2}
\begin{aligned}
& \Delta_y  \psi = \delta (y) - 1 && \text{ in } Y ,  \qquad 
& \psi \quad && Y \text{-periodic} .
\end{aligned}
\end{equation}
For similar reasons
\begin{equation} \label{eqn_outer_223c}
\begin{aligned}
\nabla_{y}\cdot (D_u \nabla_y u_{3}^O) & + 4 \pi \kappa  \nabla_y \psi \cdot \nabla_{\hat x} u_0^I
 \\ & =  \partial_t u_{1}^O - \nabla_{x}\cdot(D_u \nabla_x u_{1}^O) + 2\pi \kappa u_1^I \delta (y)  && \text{ in } Y 
\end{aligned}
\end{equation}
and $ u_{3}^O$ is $Y$-periodic. Due to the periodicity conditions imposed on $\psi$, we conclude
\begin{equation}\label{eq_u1_dep}
 \partial_t u_{1}^O =  \nabla_{x}\cdot(D_u \nabla_x u_{1}^O) - 2\pi \kappa \, u_1^I  \qquad \text{ for  }\;  x\in  \Omega_L, \; t>0.
\end{equation}
At the next order, we obtain
\begin{equation} \label{eqn_outer_4a}
\begin{aligned}
&  \nabla_{y}\cdot (D_u \nabla_y u_{4}^O) + \nabla_{y}\cdot (D_u \nabla_{\hat x} u_{3}^O) +  \nabla_{\hat x}\cdot (D_u \nabla_{y} u_{3}^O) \\
& =    \partial_t U_2^O - \nabla_{x}\cdot(D_u \nabla_x U_2^O)   +
2\pi \frac{\kappa}{D_u}  \big[\partial_t u_{0}^I - \nabla_{x}\cdot(D_u \nabla_x u_{0}^I)\big] \psi(y),    &
\end{aligned}
\end{equation}
and $ u_{4}^O$ is $Y$-periodic, and to match the contribution from the inner solution we require
\begin{eqnarray} \label{eqn_outer_4b}
 \nabla_{y}\cdot (D_u \nabla_y u_{4}^O) + \nabla_{y}\cdot (D_u \nabla_{\hat x} u_{3}^O) +  \nabla_{\hat x}\cdot (D_u \nabla_{y} u_{3}^O)
  =    \partial_t U_2^O - \nabla_{x}\cdot(D_u \nabla_x U_2^O)  \nonumber \\  +
2\pi (\kappa/D_u)  \big[\partial_t u_{0}^I - \nabla_{x}\cdot(D_u \nabla_x u_{0}^I)\big] \psi(y) + 2\pi \kappa\,  U^I_2 \delta (y)  \quad  \text{ in } Y.
\end{eqnarray}
 The solvability of \eqref{eqn_outer_4b} implies
 \begin{equation}\label{eq_DO}
 \partial_t U_2^O = \nabla_{x}\cdot(D_u \nabla_x U_2^O) - 2\pi \frac \kappa {D_u}  \big[\partial_t u_{0}^I - \nabla_{x}\cdot(D_u \nabla_x u_{0}^I)\big]  \dashint_Y\psi(y)dy - 2\pi \kappa\,  U^I_2,
\end{equation}
in  $\Omega_L$ and for $t>0$. Thus we obtain the outer approximation
\begin{equation}\label{outer_der}
\begin{aligned}
 u_0^O(t,x) + \ve u_1^O(t,x) + \ve^2 \Big(  U_2^O(t,x) + 2 \pi (\kappa/D_u) u_{0}^I(t,x) \psi(y) \Big)  + \cdots, 
 \end{aligned}
\end{equation}
 and the inner approximation
\begin{equation}\label{inner_der}
\begin{aligned}
& u_0^I(t,x) + \ve  u_1^I(t,x) + \ve^2 U_2^I(t,x) + \ve^2 (\kappa/D_u) u_0^I(t,x) \ln{(\|z\|)} +  \ve^3 U_3^I(t,x) \\
&+ \ve^3 (\kappa/D_u) u_1^I(t,x) \ln{(\|z\|)} + \ve^4 U_4^I(t,x) + \ve^4 (\kappa/D_u) U^I_2(t,x) \ln{(\| z \|)} + \cdots.
\end{aligned}
\end{equation}
Writing the latter in terms of the outer microscopic variables $y = a_\ve z$ gives 
\begin{equation}\label{inner_in_outer}
\begin{aligned}
&u_0^I(t,x) + \ve  \Big( u_1^I(t,x) + \lambda \frac{\kappa}{D_u} u_0^I(t,x) \Big) \\
&+ \ve^2 \Big( U^I_2(t,x) + \lambda \frac{\kappa}{D_u} u_1^I(t,x) +\frac{\kappa}{D_u} u_0^I(t,x) \ln{(\|y\|)} \Big) +\cdots 
\end{aligned}
\end{equation}
Comparing \eqref{outer_der} with \eqref{inner_in_outer} at $O(1)$ and $O(\ve)$ yields matching conditions
\begin{equation}\label{matching_limit_1}
\begin{aligned} 
& u_{0}^{O} (t,x)= u_{0}^{I}(t,x)= u_0(t,x), 
\\
& u_{1}^{O} (t,x)= u_{1}^{I}(t,x) + \lambda (\kappa/D_u) u_0^I(t,x) = u_1^I(t,x) + \lambda (\kappa/D_u) u_0(t,x).
\end{aligned}
\end{equation}
Matching the inner and outer solutions at $O(\ve^2)$ yields
\begin{equation}\label{matching_limit_3}
\begin{aligned}
U_2^{O} (t,x) & = U_2^I(t,x) + \lambda\frac{\kappa}{D_u} \Big[ u_1^O(t,x)- \lambda \frac{\kappa}{D_u} u_0(t,x) \Big],
\end{aligned}
\end{equation}
where we have fixed the degree of freedom in the $\psi$, satisfying \eqref{cellproblem_ve2},  by setting
\begin{equation}\label{matching_limit_4}
\begin{aligned}
\lim\limits_{y \to 0} \big\{ 2 \pi \psi(y) - \ln{(\|y\|)} \big\} = 0.
\end{aligned}
\end{equation}

Since there are no root hairs in  $\Omega\setminus \Omega_L$,  in this part of the domain the macroscopic problem  is given by the original equations.
Thus, due to the continuity of concentration and fluxes on the interface $\partial \Omega_L \setminus \partial \Omega$ between the domain with root hairs and the domain without, we  substitute \eqref{matching_limit_1} into \eqref{eqn:first_nonuni_OI} and obtain the macroscopic problem
\begin{equation} \label{eqn:first_nonuni_final}
\begin{aligned}
&\partial_t  u_{0} = \nabla_{x}\cdot(D_u \nabla_x u_{0})- 2 \pi \kappa \, u_{0}\,  \chi_{\Omega_L}&& \text{ in } \Omega, \; t>0, \\
& u_0(0, x) = u_{\rm in}(x) && \text{ in } \Omega, \\
& D_u \nabla_x u_0\cdot {\bf n }  = 0   && \text{ on } \partial \Omega\setminus \Gamma_R,  \; t>0, \\
 & D_u \nabla_x u_0\cdot {\bf n }   =  - \beta u_0  && \text{ on } \Gamma_R, \; t>0,
\end{aligned}
\end{equation}
where $\Gamma_R = \overline \Omega \cap \{ x_3=0\}$ and $\chi_{\Omega_L}$ denotes the characteristic (or indicator) function of set $\Omega_L$. Notice that we obtain the same macroscopic equation as for $u_{0,0}$  in \eqref{eqn:homo_small_a_just_CE}. This is because with $\ve \ln(1/a_\ve) = O(1)$, the term $\ve^2  \delta^{-1} u_{0,0}(t,x) \psi_{-1}^O$ from ~\eqref{final_outer_separation} is promoted to $O(\ve)$ but does not affect the leading order.

Substituting the second relation in \eqref{matching_limit_1} into \eqref{eq_u1_dep} implies the following  problem  for the first order term  $u_1(t,x) = u_1^O(t,x)$:
\begin{equation} \label{eqn:first_nonuni_final_2}
\begin{aligned}
& \partial_t  u_{1} = \nabla_{x}\cdot(D_u \nabla_x u_{1})- 2 \pi \kappa \big\{ u_1  - \lambda (\kappa/D_u) u_0 \big\}  && \text{ in } \Omega_L, \; t>0, \\
& u_1(0,x) = 0 && \text{ in } \Omega_L, \\
& D_u \nabla_x u_1\cdot {\bf n }   = 0   && \text{ on } \partial \Omega_L\setminus \Gamma_R, \; t>0, \\
 & D_u \nabla_x u_1\cdot {\bf n }   =  - \beta u_1  && \text{ on } \Gamma_R, \; t>0.
 \end{aligned}
\end{equation}
Finally, we substitute \eqref{matching_limit_3} into \eqref{eq_DO} and obtain
\begin{equation}\label{eqn:first_nonuni_final_3}
\begin{aligned}
 \partial_t U_2^O & =  \nabla_{x}\cdot(D_u \nabla_x U_2^O) +  4\pi^2 \frac{\kappa^2}{D_u}  \, u_{0}  \,   \dashint_Y\psi(y)dy  && \\
 & - 2\pi \kappa \Big( U_2^O - \lambda \frac \kappa{D_u} \Big[ u_1(t,x)- \lambda \frac \kappa{D_u} u_0(t,x) \Big] \Big)   && \text{in } \Omega_L, \; t>0, \\
 & U_2^O(0,x)   =  -2 \pi (\kappa/D_u)  u_{\rm in}(x)     \dashint_Y\psi(y)dy  &&  \text{in } \Omega_L,  \\
 & D_u \nabla_x U_2^O\cdot {\bf n }   =  - 2 \pi \kappa   \nabla_x u_0\cdot {\bf n }   \,   \dashint_Y\psi(y)dy   &&  \text{on } \partial \Omega_L\setminus \partial \Omega, \\
 &  D_u \nabla_x U_2^O\cdot {\bf n }    =  - \beta U_2^O   && \text{on } \Gamma_R, \\
 &  D_u \nabla_x U_2^O\cdot {\bf n }    =  0   && \text{on } (\partial \Omega_L\cap \partial \Omega) \setminus \Gamma_R.
\end{aligned}
\end{equation}
 Then
\begin{equation}\label{u_2_first_scale}
u_2(t,x,y)  =  U_2^O(t,x) + 2 \pi (\kappa/D_u) u_{0}(t,x) \psi(y),
\end{equation}
where $\psi$ is the solution of the `unit cell' problem \eqref{cellproblem_ve2}  satisfying  \eqref{matching_limit_4}.

For the nonlinear boundary condition \eqref{bc_hair} on the surfaces of root hairs, together with the scaling assumption \eqref{scaling_K}, we follow the same calculations as above and obtain 
\begin{equation} \label{eqn:first_nonlin_1}
\begin{aligned}
&\partial_t  u_{0} = \nabla_{x}\cdot(D_u \nabla_x u_{0})- 2 \pi \kappa \, g(u_{0})\,  \chi_{\Omega_L}&& \text{ in } \Omega, \; t>0, \\
& u_0(0, x) = u_{\rm in}(x) && \text{ in } \Omega, \\
& D_u \nabla_x u_0\cdot {\bf n }  = 0   && \text{ on } \partial \Omega\setminus \Gamma_R,  \; t>0, \\
 & D_u \nabla_x u_0\cdot {\bf n }   =  - \beta u_0  && \text{ on } \Gamma_R, \; t>0,
\end{aligned}
\end{equation}
see the Supplementary materials for the derivation. Equations for higher order approximations can be obtained in the same way as in the case of linear boundary conditions on the hair surfaces.

\subsubsection{Derivation of macroscopic equations in the case $\ve^2 \ln(1/a_\ve) = \lambda$}
\label{sec:sparse_next_nonuni}
The relation  $\ve^2 \ln(1/a_\ve) = \lambda$  is equivalent to $a_\ve = e^{-\lambda/\ve^{2}}$. The formal asymptotic expansion \eqref{eqn:ansatz_dist} used in  equations \eqref{main_1}--\eqref{bc_hair} yields 
\begin{eqnarray}
\label{outer_22_nextscaling} 
\partial_t u_{0} +  \ve \partial_t u_1 +  \cdots  =\Big[\frac 1{\ve^2}  \cA_0 + \frac 1 \ve \cA_1 + \cA_2\Big]( u_{0} + \ve u_{1} + \cdots )  \text{ in } \Omega_L\times Y_{a_\ve} ,& \\ 
\Big[\frac 1 \ve D_u \nabla_{y} + D_u \nabla_{\hat x} \Big]  ( u_{0} + \ve u_{1}  + \cdots ) \cdot \hat{\bf n }  = - \kappa e^{\frac{\lambda}{\ve^2}} \ve \left( u_{0} + \ve u_1 + \cdots \right)  \text{ on } \Omega_L\times \Gamma_{a_\ve}.& \nonumber
\end{eqnarray}
The rescaling $z=y/a_\ve$ implies
\begin{eqnarray}\label{inner_33}
 \partial_t  (u_{0} + \ve u_1 + \cdots)  = \Big[ \frac{e^{2 \lambda/ {\ve^2}} } {\ve^2} \cB_0 + \frac{ e^{\lambda/{\ve^2}} } \ve \cB_1 +  \cA_2 \Big] \left( u_{0} + \ve u_{1}   + \cdots \right)  && \;  \text{ in }  \Omega_L\times Y_{1/{a_\ve}},  \nonumber\\
\Big[ e^{\frac{\lambda}{\ve^{2}}}  \ve^{-1} D_u \nabla_{z}  + D_u \nabla_{\hat x} \Big] \left( u_0 + \ve u_{1}  + \cdots \right)  \cdot \hat{\bf n }  \qquad
\\
= -\ve \,  \kappa \, e^{\frac{\lambda}{\ve^2}} \left(u_{0}+ \ve u_1 + \cdots \right) && \;  \text{ on }  \Omega_L\times \partial B_1. \nonumber
 \end{eqnarray}
Then for the inner approximation we again obtain \eqref{eq_inner_22}. Following the same calculations as in subsection~\ref{sec:sparse_first_nonuni}, we obtain the outer approximation \eqref{outer_der} and the inner approximation \eqref{inner_der}; writing the latter in terms of the outer variables~$y$ yields
\begin{equation}\label{inner_in_outer_2}
\begin{aligned}
& \Big( u_0^I(t,x)  +  \lambda \frac{\kappa}{D_u} u_0^I(t,x) \Big) + \ve  \Big( u_1^I(t,x) + \lambda \frac{\kappa}{D_u} u_1^I(t,x) \Big)   \\
&\quad  + \ve^2  \Big( \frac{\kappa}{D_u} u_0^I(t,x) \ln{(\|y\|)} +  U_2^I(t,x) + \lambda \frac \kappa{D_u} U_2^I(t,x) \Big) + \cdots .
\end{aligned}
\end{equation}
Matching \eqref{outer_der} to \eqref{inner_in_outer_2} at $O(1)$ gives
\begin{equation}\label{matching_limit_1_1}
\begin{aligned}
u^O_{0} (t,x)= (1 + \lambda \kappa / D_u) u_{0}^{I}(t,x).
\end{aligned}
\end{equation}
Substituting  \eqref{matching_limit_1_1} into \eqref{eqn:first_nonuni_OI} yields the macroscopic problem for $u_0(t,x) = u_0^O(t,x)$:
\begin{equation} \label{eqn:homo_dist}
\begin{aligned}
&\partial_t u_{0} = \nabla_{x}\cdot(D_u \nabla_x u_{0})  - \frac{2 \pi \kappa}{1+\lambda \kappa/D_u } u_{0} \,  \chi_{\Omega_L} && \text{ in } \Omega,\; t>0, \\
& u_0(0,x) = u_{\rm in}(x) && \text{ in }  \Omega , \\
&  D_u \nabla_x u_0 \cdot {\bf n } =  - \beta u_0 \quad  &&\text{ on }   \Gamma_R, \; t>0, \\
& D_u \nabla_x u_0 \cdot {\bf n } = 0 \quad  &&\text{ on }  \partial \Omega\setminus \Gamma_R, \; t>0.
\end{aligned}
\end{equation}
Notice that \eqref{eqn:homo_dist} differs from the macroscopic equation in \eqref{eqn:homo_small_a_just_CE}, because the term $\ve^2  \delta^{-1} u_{0,0}(t,x) \psi_{-1}^O$ from ~\eqref{final_outer_separation} becomes $O(1)$ with the present scaling; for $\lambda=0$ we recover equation \eqref{eqn:homo_small_a_just_CE}, as expected.

Comparing \eqref{outer_der} with \eqref{inner_in_outer_2} at $O(\ve)$ gives
\begin{equation}\label{matching_limit_1_2}
\begin{aligned}
u_1^{O} (t,x)= (1 + \lambda \kappa/D_u) u_{1}^{I}(t,x).
\end{aligned}
\end{equation}
Substituting  \eqref{matching_limit_1_2} into \eqref{eq_u1_dep} implies that $u_1(t,x) = u_1^O(t,x)$ satisfies:
\begin{equation} \label{eqn:homo_dist_1}
\begin{aligned}
& \partial_t u_{1}  = \nabla_{x}\cdot(D_u \nabla_x u_{1}) - \frac{2 \pi \kappa}{1+\lambda  \kappa/D_u } u_{1}  && \text{ in } \Omega_L, \, t>0, \\
& u_1(0,x) = 0&& \text{ in }  \Omega_L, \\
&  D_u \nabla_x u_1 \cdot {\bf n } =  - \beta u_1 \quad  &&\text{ on }   \Gamma_R, \; t>0, \\
& D_u \nabla_x u_1 \cdot {\bf n } = 0 \quad  &&\text{ on }  \partial \Omega_L\setminus \Gamma_R, \; t >0,
\end{aligned}
\end{equation}
and we see that $u_1(t,x) = 0$ (for all $t>0$ and $x \in \Omega_L$) solves this problem. Similarly,
\begin{equation}\label{matching_limit_1_3}
\begin{aligned}
U_2^{O} (t,x)= \left( 1 + \lambda \kappa/D_u \right) U_2^I(t,x),
\end{aligned}
\end{equation}
together with condition  \eqref{matching_limit_4}  on function $\psi$. Using   \eqref{matching_limit_1_3}  in  equation \eqref{eq_DO} yields
 \begin{eqnarray}\label{eq_DO_new}
 \partial_t U_2^O   = && \,  \nabla_x \cdot(D_u  \nabla_x  U_2^O)   +  \frac \kappa{D_u}    \frac{4 \pi^2 \kappa\, u_0}{( 1+\lambda  (\kappa/D_u))^2 }    \dashint_Y \hspace{-0.1 cm } \psi(y)dy - \frac{ 2\pi \kappa}{1+ \lambda (\kappa/D_u)} U_2^O  \; \text{ in } \Omega_L, \nonumber \\
 && U_2^O(0,x)   = -  \frac{2 \pi (\kappa/D_u) }{1+\lambda (\kappa/D_u)}  u_{in}(x) \, \dashint_Y\psi(y)dy   \hspace{ 3.8 cm} \text{ in } \Omega_L,  \nonumber \\
 && D_u \nabla_x U_2^O \cdot {\bf n }   = -  \frac{2 \pi \kappa }{1+\lambda (\kappa/D_u)}  \nabla_x u_{0} \cdot {\bf n } \,  \dashint_Y\psi(y)dy  \hspace{1.8 cm} \text{ on }  \partial \Omega_L\setminus \partial \Omega,  \\
&&  D_u \nabla_x U_2^O \cdot {\bf n }   =  - \beta U_2^O  \; \text{ on } \;  \Gamma_R, \qquad   D_u \nabla_x U_2^O \cdot {\bf n }   =  0 \; \text{ on } \; ( \partial \Omega_L\cap \partial \Omega)\setminus  \Gamma_R, \nonumber
\end{eqnarray}
for $t>0$.  Hence for $u_2(t,x,y) = u_2^O(t,x,y)$ we obtain
\begin{equation}\label{u_2_second_scale}
u_2(t,x,y) =    U_2^O(t,x) + \frac{2 \pi \kappa/D_u }{1+\lambda \kappa/D_u}  u_{0}(t,x) \, \psi(y),
\end{equation}
where $\psi$ is the solution of  `unit cell' problem \eqref{cellproblem_ve2}  satisfying \eqref{matching_limit_4}.

For the nonlinear boundary condition \eqref{bc_hair} (with the scaling assumption \eqref{scaling_K}), using the Taylor expansion of  $g(u_{\ve})$  and following the same procedure as above gives
\begin{equation} \label{eqn:homo_nonlin}
\begin{aligned}
&\partial_t u_{0}  = \nabla_{x}\cdot(D_u \nabla_x u_{0}) - 2\pi \kappa \,  g(h(u_{0})) \chi_{\Omega_L} && \text{ in } \Omega, \;  t>0, \\
& D_u \nabla_x u_{0} \cdot  {\bf n}  = - \beta u_0  && \text{ on }  \Gamma_R, \; t>0 \\
& D_u \nabla_x u_{0} \cdot  {\bf n} = 0 && \text{ on }  \partial \Omega \setminus \Gamma_R, \; t>0, \\
& u_0(0,x) = u_{\rm in}(x)  && \text{ in } \Omega,
\end{aligned}
\end{equation}
where $h=h(u_0)$ is the solution of  $u_0 =h  + \lambda \,  (\kappa/D_u) g(h)$,  see the Supplementary materials for the derivation. Similar result for an elliptic problem is obtained  in \cite{gomez2015, Jaeger_small,Jaeger_small_2}. Note that by choosing $g(u)=u$ we recover the effective equation from~\eqref{eqn:homo_dist}.

Assuming  boundary condition \eqref{eqn:MM_form}, we obtain  the effective equation
\begin{equation}
\label{eqn:MM_dist_CE_explicit}
\partial_t u_{0} = \nabla_{x}\cdot(D_u \nabla_x u_{0}) - 2 \pi \kappa \frac{  \big[\sqrt{(u_0-\tilde \kappa  - 1 )^2 + 4 u_0} + u_0 - \tilde \kappa - 1   \, \big]}{2+  \big[ \sqrt{(u_0- \tilde \kappa - 1)^2 + 4 u_0} + u_0 - \tilde \kappa - 1  \, \big]} \chi_{\Omega_L},
\end{equation}
for  $x\in \Omega$,  $t>0$, and $\tilde \kappa = \lambda  \kappa/ D_u$ (see the Supplementary materials for the derivation).


\section{Rigorous derivation of macroscopic equations}\label{convergence}
In this section we give a rigorous  derivation of the macroscopic equations for \eqref{main_1}--\eqref{bc_hair}, \eqref{bc_rest},  \eqref{ic}.
To prove the convergence of solutions of multiscale problem to the solution of the corresponding macroscopic equations we first derive a priori estimates for $u_\ve$, uniform in~$\ve$. 
Due to the non-standard scale-relation  between the size and the period  of the microscopic structure considered here,  i.e.\ $a_\ve=  r_\ve/ \ve \ll 1$, we need to derive modified trace estimates and extension results, taking into account the difference in the scales between $\ve$ and $r_\ve$.  In the derivation of  the trace estimates  and extension results we follow similar ideas  as in \cite{Donato1988} with  small modifications due to the cylindrical  microstructure of~$\Omega^\ve$.

We define the   following domains, for some  $0< \rho < 1/2$, 
$$
\begin{aligned} 
& \Omega^\ve_0 =\bigcup_{\xi\in  \Xi^\ve}   \ve (\overline B_{\rho } + \xi)\times (0,L),
 \; \;  \;   \widetilde \Omega^\ve = \Omega\setminus \Omega^\ve_0, \; \;  \;  \widetilde \Omega^\ve_L = \Omega_L \setminus \Omega^\ve_0,  \;\;  \; \Omega^\ve_L = \Omega^\ve \cap \Omega_L,
\\
& \Gamma^\ve_0 = \bigcup_{\xi\in \Xi^\ve}  \ve(\partial B_{\rho } + \xi)\times (0,L), \quad  \Lambda^\ve_0 = \bigcup_{\xi\in \Xi^\ve}  \ve(\partial B_{\rho } + \xi).
\end{aligned} 
$$

\begin{lemma}\label{trace_estim}
For $v \in W^{1,p}(\Omega^\ve)$, with $1\leq p < \infty$, we have  the following trace inequality
\begin{equation}\label{trace}
\frac{\ve^2}{r_\ve} \|v\|^p_{L^p(\Gamma^\ve)} \leq \mu \left[  \|v\|^p_{L^p(\widetilde \Omega^\ve)} + \ve^p \|\nabla v\|^p_{L^p(\widetilde \Omega^\ve)} \right],  \qquad  \mu \text{- independent of }  \ve, r_\ve. 
\end{equation}
\end{lemma}
\begin{proof}
For $v \in W^{1,p}(Y_\ast\times (0,L))$ using a  trace inequality \cite{Evans_2010} in  $Y_\ast= Y \setminus \overline B_\rho$  (and an approximation of $v$ by  smooth functions) yields
\begin{equation}\label{estim_1}
\int_{\partial B_\rho} |v|^p d\gamma_{\hat y} \leq \mu_1 \int_{Y_\ast} \big( |v|^p  + |\nabla_{\hat y} v|^p \big) d\hat y,
\end{equation}
with $\hat y =(y_1, y_2)$ and  for a.a.\ $y_3 \in (0,L)$. Scaling by $r_\ve/\rho$  in the boundary integral  and by $\ve$ in the volume integral in \eqref{estim_1}  we obtain 
$$
\frac \rho {r_\ve} \int_{\partial B_{r_\ve }}  |v|^p d\hat\gamma^\ve \leq \mu_1 \frac 1{\ve^2} \int_{\ve Y_\ast} \big( |v|^p  + \ve^p  |\nabla_{\hat x} v|^p \big)d \hat x
$$
for  $x_3 \in (0,L)$, where  $\hat x = (x_1,x_2)$, $x_1 =\ve y_1$, $x_2 = \ve y_2$, $x_3 = y_3$. Adopting the changes of variables $x_j \to x_j + \ve \xi$ in the integral over $\ve Y_\ast$ and $z_j \to  z_j  + \ve \xi$ in the boundary integral,  with $j=1,2$,  and multiplying  by $\ve^2$, implies
$$
\frac {\ve^2}{r_\ve} \int_{\partial B_{r_\ve }+\ve \xi}  |v|^p d\hat\gamma^\ve \leq \mu_2  \int_{\ve Y_\ast + \ve \xi} \big( |v|^p  + \ve^p  |\nabla_{\hat x} v|^p \big)d \hat x .
$$
 Integrating the last inequality with respect to $x_3$  over $(0,L)$ and  summing up over $\xi \in \Xi^\ve$  imply the estimate~\eqref{trace}.
\end{proof}

\begin{lemma}[Extension]\label{extension_lemma}
For $v\in H^1(\Omega^\ve)$ there exists an extension $P_\ve v \in H^1(\Omega)$ such that
\begin{equation}\label{extension}
\|P_\ve v\|_{L^2(\Omega)} \leq \mu \| v\|_{L^2(\Omega^\ve)}, \quad   \|\nabla P_\ve v\|_{L^2(\Omega)} \leq \mu \|\nabla  v\|_{L^2(\Omega^\ve)}, 
\end{equation}
with a constant $\mu$ independent of $\ve$.
\end{lemma}
\begin{proof}
Consider $\tilde S = B_{2\rho}$,  $S = \tilde S \setminus \overline B_\rho$, $\tilde S_L = \tilde S \times (0,L)$,  and $S_L =S \times (0, L)$.  By a standard extension result for $v \in H^1(S\times (0,L))$ there exists $\hat v \in H^1( \tilde S  \times (0,L))$: 
\begin{equation}\label{extend_1}
\begin{aligned}
& \|\hat v\|_{L^2(\tilde S\times(0,L)) }\leq \mu_1 \|v\|_{L^2(S\times (0, L))},
\quad \|\nabla \hat v\|_{L^2(\tilde S\times (0,L))} \leq \mu_1 \|\nabla v\|_{L^2(S\times (0,L))}, \\
 & \|\nabla_{\hat x} \hat v(\cdot, x_3)\|_{L^2(\tilde S)} \leq \mu_1 \|\nabla_{\hat x} v(\cdot, x_3)\|_{L^2(S)} \quad \text{ for } \; x_3 \in (0,L) \text{ and } \hat x = (x_1, x_2), 
\end{aligned}
\end{equation}
see e.g.\ \cite{cioranescu_1999}.  Then for $v\in H^1(Y^\ve_\ast)$, where $Y^\ve_\ast = \ve Y\setminus \overline B_{r_\ve }$, consider an extension  $P_\ve: H^1(Y^\ve_\ast\times (0, L)) \to H^1(\ve Y\times (0, L))$ such that
$P_\ve v = v$ in $Y_\ast^\ve\times (0,L)$ and $P_\ve v(x) = \hat v(\rho \hat x/ r_\ve, x_3)$  in $B_{r_\ve }\times (0,L)$.
The estimates \eqref{extend_1} then give
$$
\begin{aligned}
& \int_{B_{r_\ve }\times (0,L)} \| P_\ve v\|^2 d x =  \frac{r_\ve^2}{\rho^2}   \int_{B_{\rho}\times (0,L)} \|P_\ve v\|^2 d y \leq  \frac{r_\ve^2}{\rho^2} \int_{\tilde S_L} \|P_\ve v\|^2 d y  \\
&  \leq \mu_1 \frac{ r_\ve^2}{\rho^2}  \int_{S_L} \|P_\ve v\|^2 d y
\leq \mu_1 \int_{\frac{r_\ve}\rho S\times (0, L)} \|P_\ve v\|^2 d x \leq \mu_1 \int_{Y_\ast^\ve\times (0,L)} \| P_\ve v\|^2 d x
\end{aligned}
$$
and
$$
\begin{aligned}
& \int_{B_{r_\ve }\times (0,L)} \|\nabla_{\hat x} P_\ve v\|^2 d x =  r_\ve^2 r_\ve^{-2}  \int_{B_{\rho}\times (0,L)} \|\nabla_{\hat y} P_\ve v\|^2 d y \leq   \int_{\tilde S_L} \|\nabla_{\hat y} P_\ve v\|^2 d y  \\
&  \leq \mu_1   \int_{S_L} \|\nabla_{\hat y} P_\ve v\|^2 d y
\leq \mu_1  \int_{\frac {r_\ve}\rho S\times (0,L)} \|\nabla_{\hat x} P_\ve v\|^2 d x 
 \leq \mu_1 \int_{Y^\ve_\ast\times (0,L)} \|\nabla_{\hat x}  P_\ve v\|^2 d x,
\end{aligned}
$$
where the constant $\mu_1$ is  independent of $r_\ve$ and $\ve$, and $x_j = (r_\ve/ \rho)  y_j$ for $j=1,2$, $x_3 = y_3$.
For the derivative with respect to $x_3$ we have
$$
\begin{aligned}
 & \int_{B_{r_\ve}\times (0,L)} \|\partial_{x_3} P_\ve v\|^2 d x =\frac{ r_\ve^2}{\rho^2}   \int_{B_\rho \times (0,L)} \|\partial_{y_3} P_\ve v\|^2 d y \leq  \frac{r_\ve^2}{\rho^2}   \int_{\tilde S_L} \|\partial_{y_3} P_\ve v\|^2 d y \\
&   \leq \mu_1   \frac{r_\ve^2 }{\rho^2}   \int_{S_L} \|\nabla_{y} P_\ve v\|^2 d y
\leq \mu_1 \int_{\frac{r_\ve}\rho S\times (0,L)} \|\nabla_{x} P_\ve v\|^2 d x  \leq \mu_1 \int_{Y^\ve_\ast \times (0,L)} \|\nabla_{x} P_\ve v\|^2 d x.
\end{aligned}
$$
Combining the estimates above with the fact that  $P_\ve v = v$ in $Y_\ast^\ve\times (0,L)$  yields
$$
\|P_\ve v\|_{L^2(\ve Y\times (0,L)) }\leq \mu \|v\|_{L^2(Y^\ve_\ast \times (0,L))}, \quad \|\nabla P_\ve v\|_{L^2(\ve Y \times (0,L))} \leq \mu \|\nabla v\|_{L^2(Y^\ve_\ast\times (0,L))}.
$$
Considering  the last inequalities  for  $Y^\ve_\ast + \ve \xi$ and summing up  over $\xi\in \Xi^\ve$ imply  the extension and estimates stated in lemma.
\end{proof}

\begin{lemma}  Assume   $g$ is  continuously differentiable on $[-\tilde \varsigma, \infty)$ for some $\tilde \varsigma >0$,  and $g(\eta) = g_1(\eta) + g_2(\eta)$, where  $g_1(\eta) \geq 0$ for $\eta\geq 0$, with $g_1(0) = 0$, and $g_2$ is sublinear, with $g_2(0) \leq 0$,  initial condition $u_{\rm in} \in H^1(\Omega)$, with  $0 \leq u_{\rm in } \leq u_{\rm max}$,  $K(a_\ve) = \kappa/ a_\ve$,  with $\kappa > 0$, and $\beta \geq 0$. Then   solutions  $u_\ve$ of  \eqref{main_1}--\eqref{bc_hair}, \eqref{bc_rest},  \eqref{ic} satisfy the following a priori estimates
\begin{equation}\label{apriori}
\begin{aligned}
\|u_\ve\|^2_{L^\infty(0,T; L^2(\Omega^\ve))} + \|\nabla u_\ve\|^2_{L^2((0,T)\times \Omega^\ve)}  + \beta \|u_\ve\|^2_{L^2((0,T)\times\Gamma^\ve_{R})} \\
 +\frac{\ve^2} {r_\ve} \int_{\Gamma^\ve_T} g_1(u_\ve) u_\ve \, d\gamma^\ve dt +\|\partial_t u_\ve \|^2_{L^2((0,T)\times\Omega^\ve)}&\leq \mu, 
  \\
  \|(u_\ve - M e^{m t})^+\|^2_{ L^2((0,T)\times\Omega^\ve)} &\leq \mu \ve, 
\end{aligned}
\end{equation} 
where $M, m>0$ and  the constant $\mu$ is independent of $\ve$ and of $r_\ve = \ve \, a_\ve$.
\end{lemma}
\begin{proof}
 Using assumptions on $g$ and initial data  
and employing the theorem on positive invariant sets, \cite[Theorem 2]{Redlinger1989}, we obtain  $u_\ve \geq 0$ in $\Omega^\ve_T$.  Taking  $u_\ve$ as a test function in \eqref{main_weak} and using the nonnegativity of $u_\ve$  and    assumptions on $g(u_\ve)$  ensure
\begin{equation}\label{estim_21}
\begin{aligned} 
 & \|u_\ve (s) \|^2_{L^2(\Omega^\ve)} + 2 D_u \|\nabla u_\ve\|^2_{L^2((0, s) \times \Omega^\ve)} 
+ 2 \beta \| u_\ve\|^2_{L^2((0,s)\times \Gamma^\ve_{R})}  \\
& \qquad + 2 \frac{\kappa \ve^2} {r_\ve} \int_{\Gamma^\ve_s} g_1(u_\ve) u_\ve \, d\gamma^\ve dt
\leq  \mu_1 \frac{\ve^2} {r_\ve} \|u_\ve\|^2_{L^2((0, s ) \times \Gamma^\ve)} + \mu_2
  +   \|u_\ve (0) \|^2_{L^2(\Omega^\ve)},
\end{aligned}
\end{equation}
for $s \in (0,T]$. Notice that if $g(\eta) \geq 0$ for $\eta \geq 0$, i.e.\ $g_2 \equiv 0$, we have $\mu_1 = \mu_2 =0$. 
Then using  \eqref{trace} with $p=2$ and $ \|v\|^2_{L^2(\widetilde \Omega^\ve)} \leq  \|v\|^2_{L^2(\Omega^\ve)}$,   applying  Gronwall's inequality, and taking supremum over $s \in (0,T]$, yield the first four estimates  in~\eqref{apriori}.  

Taking $(u_\ve - Me^{mt})^+$, with $M > u_{\rm max}$ and some $m>0$, as a test function in~\eqref{main_weak},  and using   assumptions on $g$ and  inequality \eqref{trace}, with $p=2$ and $p=1$, yield 
 $$
 \begin{aligned} 
 & \|(u_\ve(s) - M e^{ms})^+\|^2_{L^2(\Omega^\ve)} 
 + 2D_u \|\nabla(u_\ve - M e^{mt})^+\|^2_{L^2(\Omega^\ve_s)}
 \\
 & \quad +  2 m\| M e^{mt}(u_\ve - M e^{mt})^+ \|_{L^1(\Omega^\ve_s)} \leq  \mu_1 \|(1+ Me^{mt})(u_\ve - M e^{mt})^+ \|_{L^1(\Omega^\ve_s)} \\ 
 &\quad  + \mu_2  \|(u_\ve - M e^{mt})^+\|^2_{L^2(\Omega^\ve_s)}  + \ve (1+ M e^{ms})\big(\mu_3  \|\nabla(u_\ve - M e^{mt})^+\|^2_{L^2(\Omega^\ve_s)} + \mu_4 \big). 
 \end{aligned} 
 $$
Choosing $m$ such that $\mu_1(1+M) \leq 2 mM$ and   $\ve$ such that $\ve \mu_3 (1+Me^{mT}) \leq 2 D_u$,
and applying  Gronwall's inequality imply the last estimate in \eqref{apriori}.
\\
Taking $\partial_t u_\ve$ as a test function in \eqref{main_weak}  we obtain  
\begin{equation}\label{estim_ut} 
\begin{aligned} 
& 2 \|\partial_t u_\ve \|^2_{L^2(\Omega^\ve_s)} +D_u \|\nabla u_\ve(s) \|^2_{L^2(\Omega^\ve)}  + 
\beta \|u_\ve(s)\|^2_{L^2(\Gamma^\ve_R)} \\
& \qquad  + 2 \frac{\kappa \ve^2  } {r_\ve} \int_{\Gamma^\ve} G_1(u_\ve(s)) \, d\gamma^\ve\leq 
 \mu_1 \frac{\ve^2} {r_\ve} \|u_\ve(s)\|^2_{L^2(\Gamma^\ve)} + \mu_2
  +  \mu_3 \|u_{\rm in} \|^2_{H^1(\Omega^\ve)},
\end{aligned} 
\end{equation} 
for $s \in (0, T]$ and  $G_1(\eta) = \int_0^\eta g_1(\xi) d \xi$ for $\eta \geq 0$.  Here we used that
$$
\begin{aligned}
& \int_{\Gamma^\ve_{R,s}} \hspace{-0.2 cm } u_\ve\,  \partial_t u_\ve\,  d\gamma^\ve dt =  \frac 12 \int_{\Gamma^\ve_R} \big( |u_\ve(s)|^2 - |u_{\ve}(0) |^2 \big) d\gamma^\ve,  &&  \int_{\Gamma^\ve_R} \hspace{-0.1 cm }  |u_{\ve}(0) |^2  d\gamma^\ve 
\leq \mu_1 u^2_{\rm max}, \\
 & \int_{\Gamma^\ve_s} g(u_\ve) \partial_t u_\ve \, d\gamma^\ve dt = 
  \int_{\Gamma^\ve} \big[G(u_\ve(s)) - G(u_\ve (0)) \big] d\gamma^\ve,  && \text{where }   G(\eta) = \int_0^\eta \hspace{-0.1cm }  g(\xi) d\xi,
\end{aligned}
$$
and that  $g_1(\eta) \geq 0$ implies   $G_1(\eta)  \geq 0$ for $\eta \geq 0$, whereas the sublinearity of $g_2$ yields  $|G_2(\eta)| \leq \mu_2( |\eta|^2 + 1)$, with $G_2(\eta) = \int_0^\eta g_2(\xi) d\xi$. 
Since $u_{\rm in} \in H^1(\Omega)$  is bounded we obtain  that $u_{\rm in}$ is bounded on $\Gamma^\ve$ and $\Gamma^\ve_R$ and the continuity of  $g$ ensures that $G(u_{\rm in})$ is bounded on $\Gamma^\ve$. 
Using  \eqref{trace} with $p=2$, in 
\eqref{estim_ut}  implies the estimate for $\partial_t u_\ve$. 
\end{proof}

First we  prove convergence of a sequence of solutions of the microscopic problem for  $g(u) =u$. 
The case of a nonlinear function $g(u)$ will be considered in Theorem~\ref{th:main2_rig}.

\begin{theorem} \label{th:main1_rig}
Consider $K= \kappa/ a_\ve$ and $\ve^2 \ln(1/a_\ve) =  \lambda$ for some $\lambda >0$,  $\kappa >0$,  $\beta \geq 0$,  and initial condition $u_{\rm in} \in H^1(\Omega)$, with   $0 \leq u_{\rm in } \leq u_{\rm max}$. Then a sequence $\{ u_\ve\} $ of solutions of \eqref{main_1}--\eqref{bc_hair}, \eqref{bc_rest},  \eqref{ic} converges   to a solution $u_0 \in L^2(0,T; H^1(\Omega))$ of the macroscopic problem \eqref{eqn:homo_dist}. If $K= \kappa/ a_\ve$ and $\ve \ln(1/a_\ve) =  \lambda$ for $\lambda>0$,  then a sequence $\{ u_\ve\} $
of solutions of \eqref{main_1}--\eqref{bc_hair}, \eqref{bc_rest},  \eqref{ic} converges to a solution  $u_0 \in L^2(0,T; H^1(\Omega))$  of  the macroscopic equations
\eqref{eqn:first_nonuni_final}. 
\end{theorem}
\begin{proof}
 The a priori estimates  \eqref{apriori}  and extension Lemma~\ref{extension_lemma} imply
 $$
 \|u_\ve\|_{L^2(0,T; H^1(\Omega))} +  \|\partial_t u_\ve\|_{L^2((0,T)\times \Omega)} \leq \mu,
 $$
 with a constant $\mu$ independent of $\ve$, where $u_\ve$ is identified with its extension.  Hence there exists a function $u_0\in L^2(0,T; H^1(\Omega))$, with $\partial_t u_0 \in L^2((0,T)\times \Omega)$, such that
 \begin{equation}\label{converg_micro_u0}
 \begin{aligned} 
 &u_\ve \rightharpoonup u_0 \text{  weakly in }  L^2(0,T; H^1(\Omega)),  \; \;\;
 \partial_t u_\ve \rightharpoonup \partial_t u_0 \text{  weakly in }  L^2((0, T) \times\Omega), \\
 &u_\ve \to u_0 \text{  strongly in } L^2(0,T; H^s(\Omega)), \; \text{ for } s < 1,
 \qquad \text{(up to a subsequence)},
 \end{aligned}
 \end{equation}
  where the strong convergence is ensured by the compactness of  $H^1(\Omega) \subset H^s(\Omega)$   for~$s<1$ and the Aubin-Lions Lemma~\cite{Lions_1969}.

To pass to the limit as $\ve \to  0$ in the weak formulation of~\eqref{main_1}--\eqref{bc_hair},~\eqref{bc_rest},~\eqref{ic}  
we need to construct an appropriate  corrector  to compensate the boundary conditions on $\Gamma^\ve$.  Define $w^\ve$ to be the solution of
\begin{equation}\label{w_problem}
\begin{aligned}
& \nabla_{\hat x}\cdot(D_u \nabla_{\hat x}  w^\ve) =0  \;  && \text{ in } B_{\ve \rho } \setminus  \overline B_{r_\ve }, \\
& D_u\nabla_{\hat x}  w^\ve \cdot \hat{\bf n} = - \kappa (\ve^2/r_\ve)   w^\ve  \;  && \text{ on } \partial B_{r_\ve } , && w^\ve = 1 && \text{ on } \partial B_{\ve \rho },
\end{aligned}
\end{equation}
where $\hat x = (x_1, x_2)$, which can be solved explicitly to obtain for $\hat x \in B_{\ve \rho } \setminus \overline B_{r_\ve }$
\begin{equation}\label{sol_correct}
 w^{\ve}(\hat x) = \frac{\kappa \ve^{2}}{D_u + \kappa( \lambda  + \ve^2 \ln(\rho)) } \ln{\Big(\sqrt{x_1^2 + x_2^2}\Big)} +\frac{ D_u + \kappa (\lambda  -  \ve^{2} \ln(\ve))}{D_u+ \kappa(\lambda + \ve^2\ln(\rho))}.
 \end{equation} 
 We extend $w^\ve$ in a trivial way to $(B_{\ve \rho } \setminus  \overline B_{r_\ve })\times (0,L)$   and denote it by $\hat w^\ve(x) = w^\ve(\hat x)$.  Then we extend $\hat w^\ve(x)$
periodically with period $\ve Y$ into $\Omega^\ve \cap \Omega_0^\ve$ and  by $1$ into $\widetilde \Omega^\ve$.

Using $\phi =  \hat w^\ve \psi_1 + \psi_2$ as a test function  in \eqref{main_weak},  where  $\psi_1 \in C^1([0,T]; C^1(\overline \Omega_L))$, $\psi_2 \in C^1([0,T]; C^1(\overline{\Omega \setminus \Omega_L}))$,    with  $\psi_1(t, \hat x, L) = \psi_2(t, \hat x , L) = 0$,  and extended by zero into $\Omega_{M-L, T} =(0,T)\times(\Omega \setminus  \overline \Omega_L)$ and  $\Omega_{L, T} = (0,T)\times \Omega_{L}$ respectively,  yields
$$
\begin{aligned}
\int_{\Omega^\ve_{L,T}}  \hspace{-0.1 cm }  \Big[ \partial_t u_\ve\,  \hat w^\ve \psi_1 
  +   D_u \nabla u_\ve  \nabla (\hat w^\ve \psi_1) \Big] dx dt     +   \int_{\Gamma^\ve_T}  \hspace{-0.1 cm }  \frac{\ve^2 \kappa} {r_\ve}    u_\ve \,  \hat w^\ve \psi_1   d\gamma^\ve dt \\ +    \int_{\Gamma_{R, T}^\ve} \hspace{-0.4cm }  \beta  u_\ve \, \hat w^\ve \psi_1  d \gamma^\ve dt  
  + \int_{\Omega_{M-L,T}}  \hspace{-0.1 cm }  \Big[ \partial_t u_\ve  \psi_2 
  +   D_u \nabla u_\ve  \nabla \psi_2 \Big] dx dt = 0. 
\end{aligned}
$$ 
Notice that the assumptions on $\psi_1$ and $\psi_2$ and the construction of $\hat w^\ve$ ensure that  $\phi\in L^2(0,T; H^1(\Omega^\ve))$.  The second term in the last equality can be rewritten as
$$
\begin{aligned}
&  \int_{\Omega^\ve_{L,T}} \hspace{-0.2 cm } D_u \hat w^\ve \nabla u_\ve  \nabla \psi_1 dx  dt + \int_{\Omega^\ve_{L,T}} \hspace{-0.2 cm } D_u  \psi_1  \nabla u_\ve  \nabla \hat w^\ve  dx dt
 =  \int_{\Omega^\ve_{L,T}}\hspace{-0.2 cm } D_u \hat  w^\ve  \nabla u_\ve  \nabla \psi_1  dx dt   \\
&+ \int_{\widetilde \Omega^\ve_{L,T}} \hspace{-0.2 cm } D_u \psi_1 \nabla u_\ve   \nabla \hat w^\ve  dx dt
+  \int_{\Gamma^\ve_{T}}\hspace{-0.2 cm } D_u  u_\ve \nabla \hat  w^\ve \cdot  {\bf n} \, \psi_1 d\gamma^\ve dt   +   \int_{\Gamma_{0,T}^\ve} \hspace{-0.3 cm } D_u u_\ve \nabla \hat w^\ve\cdot {\bf  n} \, \psi_1  d\gamma^\ve dt \\
 & - \int_0^T \int_{\Omega^\ve_L\setminus \widetilde\Omega^\ve_L}
\hspace{-0.1cm } \Big[ u_\ve  \nabla\cdot(D_u \nabla \hat  w^\ve) \psi_1   + D_u u_\ve  \nabla \hat w^\ve   \nabla \psi_1  \Big] dx dt.
\end{aligned}
$$
 By the definition  of $\hat w^\ve$, we have $ \nabla\cdot(D_u \nabla \hat  w^\ve)  =0 $ in $\Omega^\ve_L\setminus \widetilde\Omega^\ve_L$  and $\nabla \hat w^\ve =0 $ in $\widetilde \Omega^\ve_L$. 
 The definition of $\hat w^\ve$ also implies 
$$\|\nabla \hat w^\ve\|_{L^2(\Omega^\ve_L)} \leq \mu,
$$
with  some constant $\mu$ independent of $\ve$.
Since $\hat w^\ve$ is bounded in $\Omega^\ve_L$,  $|\Omega_L \setminus \Omega^\ve_L| \to 0$ as $\ve \to 0$, and $\hat w^\ve =1$ in  $\widetilde \Omega^\ve_{L}$, we obtain that  $\widetilde  w^\ve \to 1$  in $L^2(\Omega_L)$  strongly, where $\widetilde w^\ve$ is the extension of $\hat w^\ve$ by zero into $\Omega_L\setminus \Omega^\ve_L$.
 Thus  strong convergence of the extension of $u_\ve$   in $L^2((0,T)\times \Omega)$  and weak convergence of $\nabla \hat w^\ve \rightharpoonup 0$ in $L^2(\Omega_L)$, using the same notation for $\hat w^\ve$ and its extension,  ensure
$$
\lim\limits_{\ve \to 0 } \int_0^T\int_{\Omega_L^\ve   \setminus\widetilde\Omega_L^\ve} D_u u_\ve  \nabla \hat w^\ve  \nabla \psi_1dx dt = 0.
$$
Using  $\|\nabla u_\ve \|_{L^2(\Omega_T)} \leq C$ and  $|\Omega \setminus \Omega^\ve| \to 0$,   $\widetilde  w^\ve \to 1$ in $L^2(\Omega_L)$,  as $\ve \to 0$,  yields
 $$
 \begin{aligned}
& \int_{\Omega^\ve_{L, T}}\hspace{-0.2 cm } \big[ \partial_t u_\ve \hat w^\ve  \psi_1 + D_u \hat w^\ve \nabla u_\ve  \nabla \psi_1\big] dx dt \to \int_{\Omega_{L,T}} \hspace{-0.2 cm } \big[\partial_t u_0 \,  \psi_1+  D_u \nabla u_0   \nabla \psi_1\big] dx dt,
\\
& \int_{\Omega_{M-L, T}} \hspace{-0.3 cm } \big[\partial_t  u_\ve \,  \psi_2+ D_u  \nabla u_\ve  \nabla \psi_2 \big] dx dt \to  \int_{\Omega_{M-L, T}} \hspace{-0.25 cm }  \big[\partial_t u_0 \, \psi_2+   D_u \nabla u_0  \nabla \psi_2 \big] dx dt,
\\
&   \int_{\Gamma_{R, T}^\ve} \beta\,  u_\ve \, \hat w^\ve \psi_1 \, d \gamma^\ve dt \to    \int_{\Gamma_{R, T}} \beta\,  u_0\,   \psi_1 \, d \hat x dt,   \hspace{3 cm }  \text{ as } \; \ve \to 0,
\end{aligned}
 $$
where the strong convergence of $u_\ve$ in $L^2(0,T; H^s(\Omega))$, for $\frac 12 < s <1$, ensures its strong convergence  in $L^2((0,T)\times \Gamma_{R})$. Computing $\nabla \hat w^\ve$  yields
$$
D_u\nabla \hat w^\ve \cdot {\bf n} =   \frac{ D_u \kappa\, \ve/\rho }{D_u+\kappa(\lambda + \ve^2\ln(\rho))}=   \frac{\kappa\, \ve/\rho }{1+ (\kappa/D_u)(\lambda + \ve^2 \ln(\rho))} \quad \text{  on } \;  \Gamma_0^\ve. 
$$
 Applying the two-scale convergence on  $\Gamma^\ve_0 = \Lambda_0^\ve \times (0,L)$,  with a test function $\psi_1 \in C^1([0,T]; C^1(\overline \Omega_L))$, 
see e.g.~\cite{Allaire_1996, Radu_1996},
and  using   $\lim\limits_{\ve \to 0} \ve \|u_\ve - u_0\|^2_{L^2(\Gamma_{0,T}^\ve)}= 0$, ensured by the strong convergence of $u_\ve$ in $L^2(0,T; H^s(\Omega))$ for $\frac 1 2 <s <1$,
see e.g.~\cite{Ptashnyk2010b},  
yields
\begin{equation}\label{two-scale_bound_1}
\begin{aligned}
\lim\limits_{\ve \to 0} \int_{\Gamma_{0,T}^\ve}\hspace{-0.3 cm } D_u  \nabla \hat  w^\ve\cdot  {\bf n}  \, u_\ve \, \psi_1  d\gamma^\ve dt   
 = \lim\limits_{\ve \to 0} \ve  \int_{\Gamma_{0,T}^\ve}\hspace{-0.1 cm }   \frac{(\kappa/\rho) \,  (u_\ve - u_0) \, \psi_1 }{1+ (\kappa/D_u)(\lambda + \ve^2 \ln(\rho))}    d\gamma^\ve  dt 
\\ + 
\lim\limits_{\ve \to 0} \ve \int_0^T \hspace{-0.15 cm } \int_0^L \hspace{-0.15 cm } \int_{\Lambda_0^\ve}\hspace{-0.1 cm }   \frac{(\kappa/\rho) \,  u_0 \, \psi_1 }{1+ (\kappa/D_u)(\lambda + \ve^2 \ln(\rho))}    d\hat\gamma^\ve dx_3 dt
\\
  =  \int_{\Omega_{L,T}}\hspace{-0.05 cm }\int_{\partial B_\rho} \frac{(\kappa/\rho) \, u_0 \, \psi_1}{1+\lambda (\kappa/D_u)}   d\hat\gamma dxdt 
=
  \int_{\Omega_{L,T}} \hspace{-0.05 cm }\frac{ 2\pi \kappa\, u_0 \, \psi_1}{1+\lambda (\kappa/D_u)}  dx dt . 
  \end{aligned}
\end{equation}
Notice that  $u_0$ and $\psi_1$ are independent of  $y \in \partial B_\rho$ and the $\ve$-scaling in the boundary integrals in~\eqref{two-scale_bound_1} is essential for the two-scale convergence on  oscillating surfaces.

Using the trace inequality  $\ve \|v\|^2_{L^2(\Gamma_{0}^\ve)} \leq \mu \| v \|^2_{H^1(\Omega_L)}$, see e.g.~\cite{Ptashnyk2010b}, we  have
$$
\begin{aligned} 
\Big|\ve  \int_{\Gamma_{0,T}^\ve}\hspace{-0.1 cm }   \frac{(\kappa/\rho) \,  (u_\ve - u_0) \, \psi_1 }{1+ (\kappa/D_u)(\lambda + \ve^2 \ln(\rho))}    d\gamma^\ve  dt\Big| & \leq \mu_1  \ve^{\frac 1 2} \|u_\ve - u_0\|_{L^2(\Gamma_{0,T}^\ve)}\| \psi_1 \|_{L^2(0,T; H^1(\Omega_L))},\\
\ve \Big\| \frac{(\kappa/\rho) \,  u_0 }{1+ (\kappa/D_u)(\lambda + \ve^2 \ln(\rho))} \Big\|^2_{L^2(\Gamma_{0,T}^\ve)} & \leq \mu_2 \|u_0\|^2_{L^2(0,T; H^1(\Omega_L))} \leq \mu_3, 
\end{aligned} 
$$
for  $0<\ve \leq \ve_0$, such that $\lambda + \ve_0^2 \ln(\rho) >0$ with $0< \rho < 1/2$.

Combining all the calculations from above, in the limit as $\ve \to 0$, we obtain 
  the equation and boundary conditions  in  \eqref{eqn:homo_dist}. Standard arguments, see e.g.~\cite{Ptashnyk2010}, ensure that $u_0$  satisfies the initial condition in \eqref{eqn:homo_dist} and   is a unique solution of \eqref{eqn:homo_dist}. Hence the whole sequence $\{u_\ve\}$ converges to $u_0$ as $\ve \to 0$. 

If $\ve \ln(1/a_\ve) = \lambda$ then the solution of problem \eqref{w_problem}  is given by
\begin{equation}\label{sol_correct_2}
\begin{aligned} 
& w^{\ve}(x_1, x_2) = \frac{\kappa \ve^{2}}{D_u +\kappa( \ve \lambda +\ve^2 \ln(\rho)) } \ln{\Big(\sqrt{x_1^2 + x_2^2}\Big)} + \frac{ D_u + \kappa (\ve\lambda  -  \ve^{2} \ln(\ve))}{D_u+ \kappa(\ve\lambda + \ve^2\ln(\rho))}, \\
 & D_u\nabla \hat w^\ve \cdot {\bf n}  =  \ve \, \frac{ \kappa /\rho }{1+  (\kappa/D_u)(\ve \lambda + \ve^2\ln(\rho))}  \quad \text{  on } \;  \Gamma_0^\ve. 
 \end{aligned} 
 \end{equation}
In this case the boundary integral converges to
$$
\int_0^T\int_{\Gamma_0^\ve} D_u \nabla \hat  w^\ve\cdot {\bf n} \, u_\ve \psi_1 \, d\gamma^\ve  dt  \to \int_0^T    \int_{\Omega_L}  2\pi \kappa \,  u_0\,  \psi_1\,  dx dt \quad \text{ as } \ve \to 0, 
$$
and we obtain the macroscopic equation as in \eqref{eqn:first_nonuni_final}.
\end{proof}
Now we consider the nonlinear  condition \eqref{bc_hair}  on the boundaries  of the microstructure. 
\begin{theorem} \label{th:main2_rig}
Consider $K= \kappa/ a_\ve$, for $\kappa>0$,  and $\ve^2 \ln(1/a_\ve) =  \lambda$ for some $\lambda >0$,  let $g$  be  continuously differentiable  and monotone non-decreasing on $[-\tilde \varsigma, \infty)$, for some $\tilde \varsigma>0$,   and $g(\eta) = g_1(\eta) + g_2(\eta)$, where  $g_1(\eta) \geq 0$ for $\eta\geq 0$, with $g_1(0) = 0$, and $g_2$ is sublinear, with $g_2(0) \leq 0$,  initial condition $u_{\rm in} \in H^1(\Omega)$ with  $0 \leq u_{\rm in} \leq u_{\rm max}$,  and $\beta \geq 0$. Then a sequence $\{ u_\ve\} $ of solutions of \eqref{main_1}--\eqref{bc_hair}, \eqref{bc_rest},  \eqref{ic}  converges to a solution $u_0 \in L^2(0,T; H^1(\Omega))$ of the macroscopic problem \eqref{eqn:homo_nonlin}. If $K= \kappa/ a_\ve$ and $\ve \ln(1/a_\ve) =  \lambda$ for $\lambda>0$ then a sequence $\{ u_\ve\} $ of solutions of \eqref{main_1}--\eqref{bc_hair}, \eqref{bc_rest},  \eqref{ic}   converges to a solution $u_0 \in L^2(0,T; H^1(\Omega))$ of the macroscopic equations~\eqref{eqn:first_nonlin_1}. 
\end{theorem}
\begin{proof} 
In the same way as in the proof of  Theorem~\ref{th:main1_rig}, using  a priori estimates~\eqref{apriori}  and extension Lemma~\ref{extension_lemma} we obtain 
 following convergence results 
  \begin{equation}\label{converg_micro_u0_2}
  \begin{aligned} 
  & u_\ve \rightharpoonup u_0 \text{  weakly in }  L^2(0,T; H^1(\Omega)),  \;  \partial_t u_\ve \rightharpoonup \partial_t u_0 \text{  weakly in }  L^2((0,T)\times\Omega), \\
  & u_\ve \to u_0 \text{  strongly in } L^2(0,T; H^s(\Omega)), \;  \text{for } \; s <1, \quad \text{ (up to a subsequence)},
  \end{aligned}
  \end{equation}
where $u_0 \in L^2(0,T; H^1(\Omega))\cap H^1(0,T; L^2(\Omega))$. Since $u_\ve \geq 0$ for all $\ve >0$ we have $u_0 \geq 0$, whereas the last estimate in~\eqref{apriori}, together with the strong convergence of~$u_\ve$, implies $u_0 \in L^\infty((0,T)\times \Omega)$. 
 
 As in the proof of Theorem~\ref{th:main1_rig}, the main step  is to construct an appropriate corrector to pass to the limit in the  integral over the boundaries of the microstructure.
In a similar  way as in \cite{gomez2015, Jaeger_small_2},  we define $w^\ve$ to be the solution of 
\begin{equation}\label{corrector2}
\begin{aligned}
& \Delta w^\ve  = 0 \;   \text{ in }\;  B_{\ve\rho} \setminus \overline B_{r_\ve}, \; \; 
&&   w^\ve  =  1  \;     \text{ on } \; \partial B_{r_\ve} ,  \; \; 
&& w^\ve = 0   \;   \text{ on } \; \partial B_{\ve\rho}.
\end{aligned}
\end{equation}
Then we   extend $w^\ve$ by $1$ into $B_{r_\ve}$, in a trivial way into the $x_3$-direction for $x_3\in(0,L)$, by $w^\ve(\hat x)[1+ (L-x_3)/\ve]$ for $x_3\in [L, L+\ve)$, and then $\ve Y$-periodically into $\Omega_0^\ve\cup \Omega^\ve_{0, L+\ve}$, where $\Omega^\ve_{0, L+\ve} = \bigcup_{\xi \in \Xi^\ve} \ve(\overline B_\rho + \xi) \times [L, L+\ve)$,   and by $0$ into $\widetilde\Omega^\ve_{L+\ve} = \widetilde\Omega^\ve \setminus \Omega^\ve_{0, L+\ve}$. We denote this extension of $w^\ve$ again by $w^\ve$. 
Then $w^\ve(x) =\ln(|\hat x|/(\ve\rho))\big[\ln(r_\ve/(\ve \rho))\big]^{-1}$ for $x \in  \Omega^\ve \cap \Omega_0^\ve $ and $w^\ve(x) = 0$ for $x\in \widetilde \Omega^\ve_{L+\ve}$.  The assumption  on the relation between~$\ve$ and $a_\ve = r_\ve / \ve$ implies
$$
\begin{aligned}
&  \int_{\Omega^\ve_L \setminus \widetilde \Omega^\ve} \hspace{-0.2 cm } |\nabla w^\ve |^2 dx =  \frac 1{\ln(\ve \rho/r_\ve)^{2}} \int_{\Omega^\ve_L \setminus \widetilde \Omega^\ve} \hspace{-0.1 cm } \frac 1 {| \hat{x} |^2} dx \le \frac {2\pi \mu_1 L}{\ve^2\ln(\ve \rho/r_\ve)^{2}}  \int_{r_\ve}^{\ve \rho}    \frac {dr} r     \le  \mu, 
\\
&  \int_{\Omega^\ve_{0,L+\ve}} \hspace{-0.2 cm } |\nabla w^\ve |^2 dx \le  \mu_1 \ve \|\nabla  w^\ve \|^2_{L^2(\Omega^\ve_L \setminus \widetilde \Omega^\ve)} + 
 \frac {\mu_2}{ \ve}  \|w^\ve \|^2_{L^2(\Omega^\ve_L \setminus \widetilde \Omega^\ve)}  \le \mu\, \ve,  
 \end{aligned}
$$
 for some constant $\mu >0$ independent of $\ve$. This, together  with similar arguments as in Theorem~\ref{th:main1_rig}, implies that   $w^\ve \rightharpoonup 0$ weakly in $H^1(\Omega)$ and  strongly in $H^s(\Omega)$ for $s < 1$.

To prove convergence of solutions of problem  \eqref{main_1}--\eqref{bc_hair}, \eqref{bc_rest},  \eqref{ic},   by using the monotonicity of $g$, we  rewrite its weak formulation \eqref{main_weak} as variational inequality
\begin{equation} \label{ine_micro}
\begin{aligned}
\int_{\Omega^\ve_T}\Big[ \partial_t u_\ve ( \phi - u_\ve) + D_u\nabla \phi \nabla(\phi - u_\ve) \Big]dx dt +  \frac{\ve^2 \kappa }{r_\ve}  \int_{\Gamma^\ve_T}  g(\phi) (\phi - u_\ve) d\gamma^\ve dt \\
+ \int_{\Gamma^\ve_{R,T}} \beta \,  \phi \, ( \phi - u_\ve) d\gamma^\ve dt  
\ge 0  
\end{aligned} 
\end{equation}
for any $\phi \in L^2(0,T; H^1(\Omega^\ve))\cap L^\infty((0,T)\times\Omega^\ve)$, with $ \phi(t,x) \geq - \tilde \varsigma$ in $(0,T)\times \Omega^\ve$. Notice that the last condition on $\phi$ is not needed if $g$ is monotone on $\mathbb R$. 

Considering  $\phi = \psi -  \tilde \kappa g(h) w^\ve$, for $\psi \in C^1([0,T]; C^1(\overline\Omega))$ 
with $\psi(t,x) \geq -\tilde \varsigma$ in $[0,T]\times\overline \Omega$,  as a test function in~\eqref{ine_micro}, where $\tilde \kappa = \lambda \kappa/D_u$  and $h$ is  the solution of $h+ \tilde \kappa g(h) = \psi$, and using the weak and strong convergence   of $w^\ve$  and  of extension of $u_\ve$, in the corresponding spaces, together with  $|\Omega \setminus \Omega^\ve| \to 0$ as $\ve \to 0$,   we obtain 
$$
\begin{aligned}
& \lim\limits_{\ve \to 0 }  \int_{\Omega^\ve_T} \partial_t u_\ve (\psi -  \tilde \kappa g(h) w^\ve - u_\ve) dxdt =  \int_{\Omega_T} \partial_t u_0 (\psi - u_0) dxdt, \\
&\lim\limits_{\ve \to 0 } \int_{\Gamma^\ve_{R,T}} \beta (\psi - \tilde \kappa g(h) w^\ve) ( \psi -\tilde  \kappa g(h) w^\ve - u_\ve) d\gamma^\ve dt
= \int_{\Gamma_{R,T}} \beta \, \psi ( \psi - u_0) d\hat x dt.
\end{aligned}
$$
Here and in what follows we use the same notation for   $u_\ve$  and its extension. For the second term in \eqref{ine_micro}, the weak convergence of $\nabla u_\ve$  and   $|\Omega \setminus \Omega^\ve| \to 0$, as $\ve \to 0$, yield 
$$
\begin{aligned} 
\lim\limits_{\ve \to 0 } \int_{\Omega^\ve_T}  D_u\nabla ( \psi - \tilde \kappa g(h) w^\ve) \nabla( \psi -\tilde  \kappa g(h) w^\ve - u_\ve) dx dt
= \int_{\Omega_T} D_u \nabla \psi \nabla (\psi - u_0) dx dt 
\\ - \lim\limits_{\ve \to 0 } \int_{\Omega^\ve_T} D_u \tilde \kappa (\nabla g(h) w^\ve +  g(h) \nabla  w^\ve) \nabla( \psi - \tilde \kappa g(h) w^\ve - u_\ve) dx dt. 
\end{aligned} 
$$
For the first part of the  last term the strong convergence of $w^\ve$  and weak convergence of $\nabla w^\ve $ and $\nabla u_\ve$ in $L^2(\Omega_T)$ ensure
$$
\begin{aligned} 
 \lim\limits_{\ve \to 0 } \int_{\Omega^\ve_T} D_u  \tilde \kappa \nabla g(h) w^\ve  \nabla( \psi -\tilde  \kappa g(h) w^\ve - u_\ve) dx dt = 0,  
\end{aligned} 
$$
and the second part can be rewritten as 
 $$
\begin{aligned} 
 \int_{\Omega^\ve_T}  \hspace{-0.1 cm } D_u\tilde \kappa \big[\nabla  w^\ve  \nabla \big(g(h)[ \psi -\tilde  \kappa g(h) w^\ve - u_\ve]\big)   - 
   \nabla  w^\ve  \nabla g(h)( \psi - \tilde \kappa g(h) w^\ve - u_\ve) \big] dx dt  \\= I_1 + I_2,  
\end{aligned} 
$$
where $\lim\limits_{\ve \to 0} I_2 = 0$, due to weak convergence of $\nabla w^\ve$ and strong convergence of $u_\ve$ and $w^\ve$ in $L^2(\Omega_T)$.  Using that 
$\Delta w^\ve =0$ in $\Omega^\ve \cap \Omega^\ve_0$  and $\nabla w^\ve =0$ in $\Omega^\ve\setminus (\Omega_0^\ve \cup \Omega_{0, L+\ve}^\ve)$ and  integrating by parts in $I_1$ yield 
$$
I_1 = \frac{\lambda \kappa}{\lambda + \ve^2 \ln(\rho)}  \Big[  \frac{\ve^2}{r_\ve}  \int_{\Gamma_{T}^\ve}  \hspace{-0.1 cm }   g(h)( \psi - \tilde \kappa g(h)- u_\ve) d\gamma^\ve dt - \frac{\ve} {\rho} \int_{\Gamma_{0, T}^\ve} \hspace{-0.1 cm }   g(h)( \psi -  u_\ve) d\gamma^\ve dt \Big]+ I_{11}, 
$$
where, due to   $\lim\limits_{\ve \to 0 } \|\nabla w^\ve\|_{L^2(\Omega^\ve_{0, L+\ve})} = 0$, we have 
$$
I_{11} = \int_0^T\int_{\Omega^\ve_{0, L+\ve}} D_u\tilde \kappa \nabla  w^\ve  \nabla (g(h)[ \psi -\tilde  \kappa g(h) w^\ve - u_\ve]) dx dt \to 0\; \; \text{ as } \; \; \ve \to 0.  
$$
 Similar as in the proof of Theorem~\ref{th:main1_rig}, using 
the two-scale convergence on $\Gamma_0^\ve$, see e.g.~\cite{Allaire_1996, Radu_1996}, and that  $\lim\limits_{\ve \to 0 }\ve \| u_\ve - u_0\|^2_{L^2(\Gamma^\ve_{0,T})} =0$, see e.g.~\cite{Ptashnyk2010b}, we obtain
$$
\begin{aligned} 
\lim\limits_{\ve \to 0} \ve \frac{\lambda (\kappa/\rho)}{\lambda + \ve^2 \ln(\rho)} 
\int_{\Gamma_{0, T}^\ve} \hspace{-0.2 cm }   g(h)( \psi -  u_\ve) d\gamma^\ve dt 
= \lim\limits_{\ve \to 0}    \frac{\lambda (\kappa/\rho)}{\lambda + \ve^2 \ln(\rho)} \,  \ve  \int_{\Gamma_{0, T}^\ve} \hspace{-0.2 cm }  g(h)(u_0- u_\ve) d\gamma^\ve dt 
\\ +  \lim\limits_{\ve \to 0}   \frac{\lambda (\kappa/\rho)}{\lambda + \ve^2 \ln(\rho)} \,   \ve \int_{\Gamma_{0, T}^\ve} \hspace{-0.2 cm }  g(h)( \psi -  u_0) d\gamma^\ve dt 
 = 2\pi \kappa \int_{\Omega_{L,T}} \hspace{-0.2 cm } g(h) ( \psi - u_0) dx dt.
\end{aligned} 
$$  
Notice that the regularity $g(h) \in C^1([0,T]; C^1(\overline \Omega))$, ensured by the regularity of $g$ and~$\psi$, and  the trace estimate  $\ve \|v \|^2_{L^2(\Gamma_{0}^\ve)}\leq \mu \| v \|^2_{H^1(\Omega_L)}$, see e.g.~\cite{Ptashnyk2010b}, yield   
 $$
\begin{aligned} 
  & \Big|\frac{\lambda (\kappa/\rho)}{\lambda + \ve^2 \ln(\rho)} \,  \ve  \int_{\Gamma_{0, T}^\ve} \hspace{-0.3 cm }  g(h)(u_0- u_\ve) d\gamma^\ve dt\Big| \leq \mu_1 \ve^{\frac 1 2} \|u_0- u_\ve\|_{L^2(\Gamma_{0, T}^\ve)} \| g(h)\|_{L^2(0, T; H^1(\Omega))}, \\
  & \ve\Big\| \frac{\lambda (\kappa/\rho)}{\lambda + \ve^2 \ln(\rho)} ( \psi -  u_0) \Big\|^2_{L^2(\Gamma_{0, T}^\ve)} \leq  \mu_2 \big[ \|u_0\|^2_{L^2(0,T; H^1(\Omega))} +  \|\psi\|^2_{L^2(0,T; H^1(\Omega))} \Big] \leq \mu_3,
\end{aligned}
$$
for  $0<\ve \leq \ve_0$, with $\lambda + \ve_0^2 \ln(\rho)>0$ and $0< \rho < 1/2$. 
It remains  to show  that 
$$
 \frac{\kappa\ve^2}{r_\ve}   \int_{\Gamma^\ve_T} \Big(g( \psi - \tilde  \kappa g(h) )  - \frac{\lambda }{\lambda + \ve^2 \ln(\rho)}  g(h)\Big)[ \psi -  \tilde  \kappa g(h)  - u_\ve] d\gamma^\ve dt \to 0  \; \; \text{as } \; \ve \to 0. 
 $$
 Since $h$ is the solution of $h + \tilde \kappa g(h)= \psi$ and $g$ is monotone and continuous we have 
 $$
 \frac{\kappa\ve^2}{r_\ve}   \int_{\Gamma^\ve_T} [g( \psi - \tilde  \kappa g(h) )  -  g(h)][ \psi -  \tilde \kappa g(h)  - u_\ve] d\gamma^\ve dt =0. 
 $$ 
The trace estimate  \eqref{trace} yields  
 $$
 \begin{aligned} 
& \Big[\frac{\lambda }{\lambda + \ve^2 \ln(\rho)} - 1 \Big] \frac{\kappa\ve^2}{r_\ve}   \int_{\Gamma^\ve_T} | g(h)|
 | \psi -  \tilde \kappa g(h)  - u_\ve| d\gamma^\ve dt   \le \mu  \Big[\|h\|^2_{L^2(0,T; H^1(\widetilde \Omega^\ve_L))} \\ 
 & \quad + \|\psi\|^2_{L^2(0,T; H^1(\widetilde \Omega^\ve_L))} + \|u_\ve\|^2_{L^2(0,T; H^1(\widetilde \Omega^\ve_L))} +1 \Big]\Big[ \frac{\lambda }{\lambda + \ve^2 \ln(\rho)} -1 \Big]\to 0,
 \; \text{ as } \ve \to 0.
 \end{aligned}  
 $$

Collecting all calculations from above, taking the limit as $\ve \to 0$ in  \eqref{ine_micro},   with $\phi = \psi - \tilde \kappa g(h) w^\ve$, and employing  a density argument, we obtain 
\begin{equation} \label{ine_macro}
\begin{aligned}
\int_{\Omega_T} \big[\partial_t u_0 ( \psi - u_0) + D_u\nabla \psi \nabla(\psi - u_0) \big] dx dt 
+   \int_{\Omega_{L,T}} 2 \pi \kappa g(h) (\psi - u_0) dx dt \\
+ \int_{\Gamma_{R,T}} \beta \,  \psi \, ( \psi - u_0) d\hat x dt
\ge 0  
\end{aligned} 
\end{equation}
for any $\psi \in L^2(0,T; H^1(\Omega))\cap L^\infty((0,T)\times \Omega)$. 
 By choosing $\psi = u_0  \pm \sigma \varphi $, for $\sigma >0$ and $\varphi \in L^2(0,T; H^1(\Omega))\cap L^\infty((0,T) \times \Omega)$, and letting $\sigma \to 0$ we obtain that $u_0$ is a solution of the macroscopic problem \eqref{eqn:homo_nonlin}. Since $u_0 \geq 0$ we have $\psi \geq - \tilde \varsigma$ for sufficiently small~$\sigma$. Standard calculations ensure uniqueness of a solution of \eqref{eqn:homo_nonlin}.  

If $K= \kappa/ a_\ve$ and $\ve \ln(1/a_\ve) =  \lambda$, we again rewrite~\eqref{main_1}--\eqref{bc_hair},~\eqref{bc_rest},~\eqref{ic} as variational  inequality~\eqref{ine_micro}. The convergence,  as $\ve \to 0$, of the first two terms and of the last integral in \eqref{ine_micro} follows directly from the weak convergence  $u_\ve \rightharpoonup u_0 $ in $L^2(0,T; H^1(\Omega)) \cap H^1(0,T; L^2(\Omega))$ and $|\Omega\setminus \Omega^\ve| \to 0$ as $\ve \to 0$. 
To show  
\begin{equation}\label{converg_bound_22}
\lim\limits_{\ve \to 0} \frac{\ve^2 \kappa }{r_\ve}  \int_{\Gamma^\ve_T}  g(\phi) (\phi - u_\ve) d\gamma^\ve dt =   2 \pi \kappa  \int_{\Omega_{L,T}}  g(\phi) (\phi - u_0) dx dt 
\end{equation}
we consider  the solution of the following problem  
$$
\nabla\cdot(D_u \nabla \tilde w^\ve) = 0 \text{ in }  B_{\ve\rho} \setminus \overline B_{r_\ve}, \; \;  D_u \nabla \tilde w^\ve \cdot \nu = \frac {\ve^2 \kappa }{r_\ve} \text{ on } \partial B_{r_\ve},  \; \;  
\tilde w^\ve = 0 \text{ on } \partial B_{\ve \rho},  
$$
given by $\tilde w^\ve = \ve^2 (\kappa/D_u) \ln(|\hat x|/(\ve \rho))$,  extended in a trivial way to $(B_{\ve\rho} \setminus \overline B_{r_\ve})\times (0,L)$ and then  $\ve Y$- periodically into $\Omega^\ve\cap\Omega^\ve_0$. Notice  $ |\tilde w^\ve(x)| \leq (\kappa/ D_u) \ve^2 \ln(\ve\rho/ r_\ve)   \le \mu \,  \ve$, for all $x\in \Omega^\ve\cap \Omega^\ve_0$, and 
$$
\int_{\Omega^\ve\cap\Omega_0^\ve} |\nabla \tilde w^\ve|^2 dx \le  \mu_1  \ve^2 
 \int_{r_\ve}^{\ve \rho} \frac 1 r dr  \le \mu\, \ve,  
$$
with a constant $\mu>0$ independent of $\ve$. Then 
$$
\begin{aligned} 
0=  -  \int_0^T \hspace{-0.1 cm }  \int_{\Omega^\ve\cap\Omega^\ve_{0}}\hspace{-0.5 cm}  \nabla\cdot (D_u \nabla \tilde w^\ve ) g(\phi) (\phi - u_\ve)  dx dt 
=\int_0^T  \hspace{-0.1 cm }  \int_{\Omega^\ve\cap\Omega_0^\ve}  \hspace{-0.5 cm}  D_u \nabla \tilde w^\ve \nabla \big[g(\phi) (\phi - u_\ve)  \big] dx dt \\ 
+  \frac{\ve^2 \kappa }{r_\ve}  \int_{\Gamma^\ve_T}  g(\phi) (\phi - u_\ve) d\gamma^\ve dt   - \ve  \frac{ \kappa}\rho  \int_{\Gamma^\ve_{0,T}}  g(\phi) (\phi - u_\ve) d\gamma^\ve dt. 
 \end{aligned} 
$$
 Hence taking  in the last equality the limit as $\ve \to 0$ and using weak convergence of $u_\ve$ in $L^2(0,T; H^1(\Omega))$ and two-scale convergence on $\Gamma^\ve_0$, together with  the fact  that 
 $\lim\limits_{\ve \to 0}\|\nabla \tilde w^\ve \|_{L^2(\Omega^\ve\cap \Omega^\ve_0)} = 0$, 
 imply \eqref{converg_bound_22}. 
 By choosing $\phi = u_0  \pm \sigma\varphi $, for $\sigma >0$ and $\varphi \in L^2(0,T; H^1(\Omega))\cap L^\infty((0,T) \times \Omega)$, and letting $\sigma \to 0$ we obtain that $u_0$ is the solution of the macroscopic problem~\eqref{eqn:first_nonlin_1}.
 Notice that in the case $\ve \ln(1/a_\ve) =  \lambda$ we can also show convergence of solutions of   \eqref{main_1}--\eqref{bc_hair},~\eqref{bc_rest},~\eqref{ic} directly, without rewriting it as a variational inequality and using monotonicity of $g$. 
\end{proof}

\section{Numerical simulations for multiscale and macroscopic models}
\label{sec_numerics}
In this section we present numerical simulations of \eqref{main_1}--\eqref{bc_hair}, \eqref{bc_rest},  \eqref{ic} and of the zero, first and second order approximations of solutions of the macroscopic problems, see \eqref{eqn:homo_dist},  \eqref{eqn:homo_dist_1}, \eqref{eq_DO_new}. All simulations in this section were performed using standard finite element methods as implemented in FEniCS \cite{logg2011}, with meshed domains generated using NETGEN \cite{schoberl1997}. Steady-state (elliptic) problems were solved directly, while for time-dependent (parabolic) problems, backwards Euler discretization in time was used and the solution at time $t+ \Delta t$ was calculated using the  stationary solver with the solution at time $t$ entering the right-hand side of the weak formulation as a given forcing term (as described in \cite{logg2011}). Since the scale-relation $\ve^{2} \ln{\big(1/a_\ve \big)} = \lambda$ for small $\ve$ results in a very small value for $a_\ve$, which is numerically challenging, we consider (only) $\ve = 0.5$ and observe that $a_\ve = 0.01$ with such $\ve$ gives $\lambda = \ve^{2} \ln{\big(1/a_\ve \big)} \approx 1.15$. Continuous Galerkin finite element method of degree $1$ was used and tetrahedral meshes for the full-geometry simulations were created using in-built NETGEN generators with automatic mesh refinement close to the root hair, so that the size of any tetrahedron does not exceed $0.03$, which in the 
case of $a_\ve = 10^{-3}$ (see below) yielded $O(7 \times 10^5)$ tetrahedra. For the macroscopic problems in our two-scale expansions (i.e. $u_0$, $u_1$ and $U_2$), we generated meshes with the maximum mesh size of $0.05$, which yielded $O(14000)$ tetrahedra for the  mesh for domain $\Omega$, and $O(7000)$ for the mesh for domain $\Omega_L$.

We first consider the steady-state problem for  equation \eqref{main_1}, imposing a constant level of nutrient at the cut-off distance
\begin{equation}\label{dirichlet_bc_ff}
u_\ve(t,x) = 1  \quad  \; \;  \text{ on } \;   x_{3} = M, \; t>0,
\end{equation}
and a zero-flux boundary condition on $\partial \Omega \setminus \{ x_3 = M\}$, i.e.\ $\beta = 0$.
Then in the corresponding macroscopic problem we have
$$
 u_0(t,x) = 1\; \; \;  \text{ on } \;   x_{3} = M,  \quad D_u \nabla u_0(t,x) \cdot {\bf n} = 0 \; \; \; \text{ on } \partial \Omega \setminus \{ x_3 = M\}, \; \; t>0.
 $$
Notice that the choice of boundary condition on $x_{3} = M$ does not affect the derivations of macroscopic equations in Sections~\ref{formal_derivation} and~\ref{convergence}.
 The symmetries of the full-geometry problem and the periodicity of the microstructure ensure that the solution of this problem has the same behavior in each periodicity cell $\ve(Y + \xi)\times (0,M)$, for $\xi \in \mathbb Z^2$, see Figure~SM1 
 in the Supplementary materials. Hence it is sufficient to determine the solution within a single periodicity  cell $\ve Y\times (0,M)$.

To illustrate the differences in the behavior of the multiscale solutions and those of the corresponding macroscopic problems  \eqref{eqn:first_nonuni_final} and  \eqref{eqn:homo_dist}
for two different scale-relations between $\ve$ and $a_\ve$, we vary $a_\ve$ from $10^{-1}$ to $10^{-3}$, see Figure \ref{fig:comparison_dist_limit}. The default parameter values used throughout this section are summarized in Table~\ref{table_1}.
\begin{figure}
\subfloat[$u_0$ for $\ve \ln{(1/a_\ve)} = O(1)$]{\includegraphics[width=0.5\textwidth]{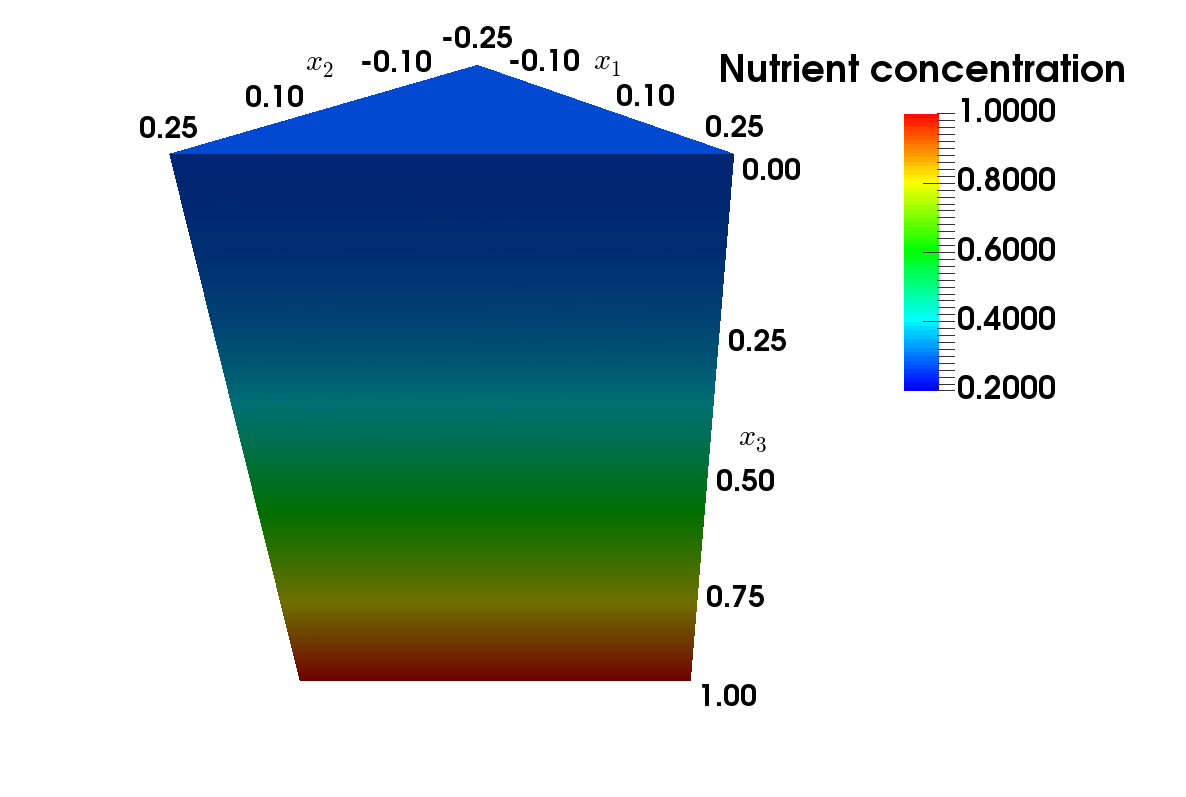}}
\subfloat[$u_\ve$ for $a_\ve=10^{-1}$ ($\ve^{2} \ln{(1/a_\ve)} \approx 0.58$)]{\includegraphics[width=0.5\textwidth]{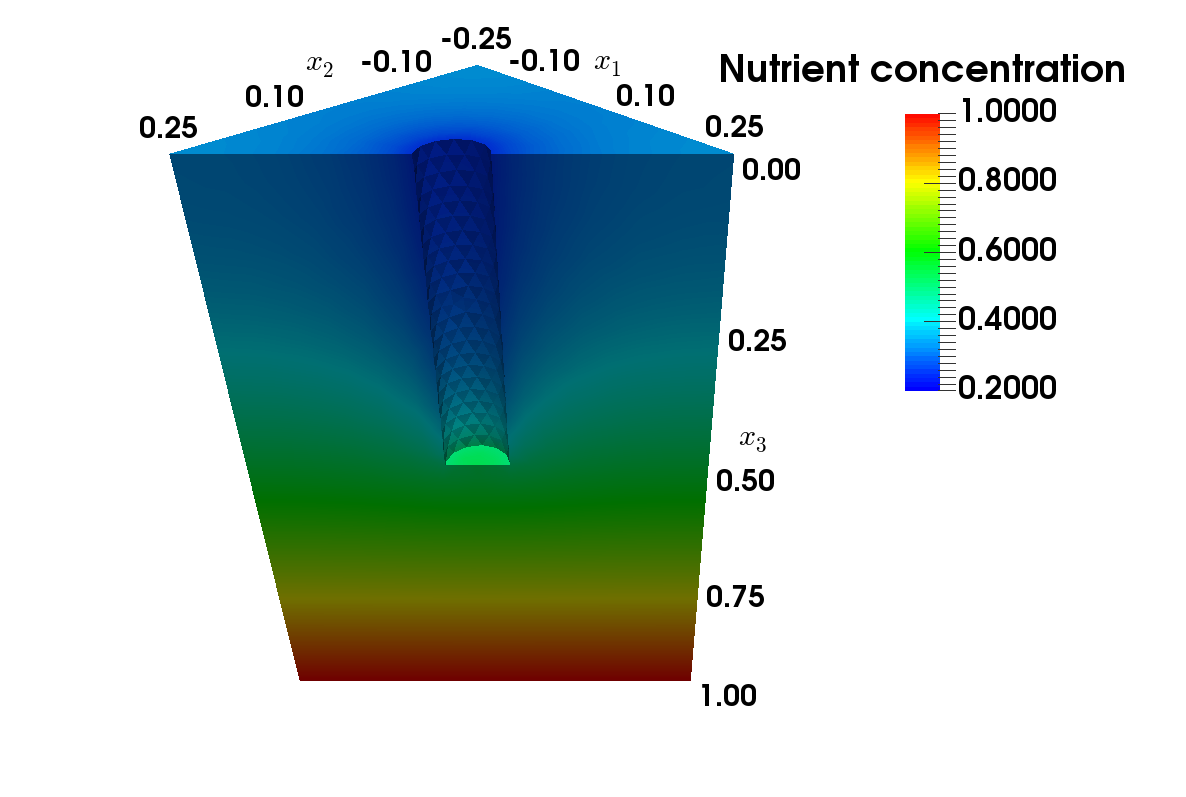}}
\\
\subfloat[$u_\ve$ for $a_\ve=10^{-2}$ ($\ve^{2} \ln{(1/a_\ve)} \approx 1.15$)]{\includegraphics[width=0.5\textwidth]{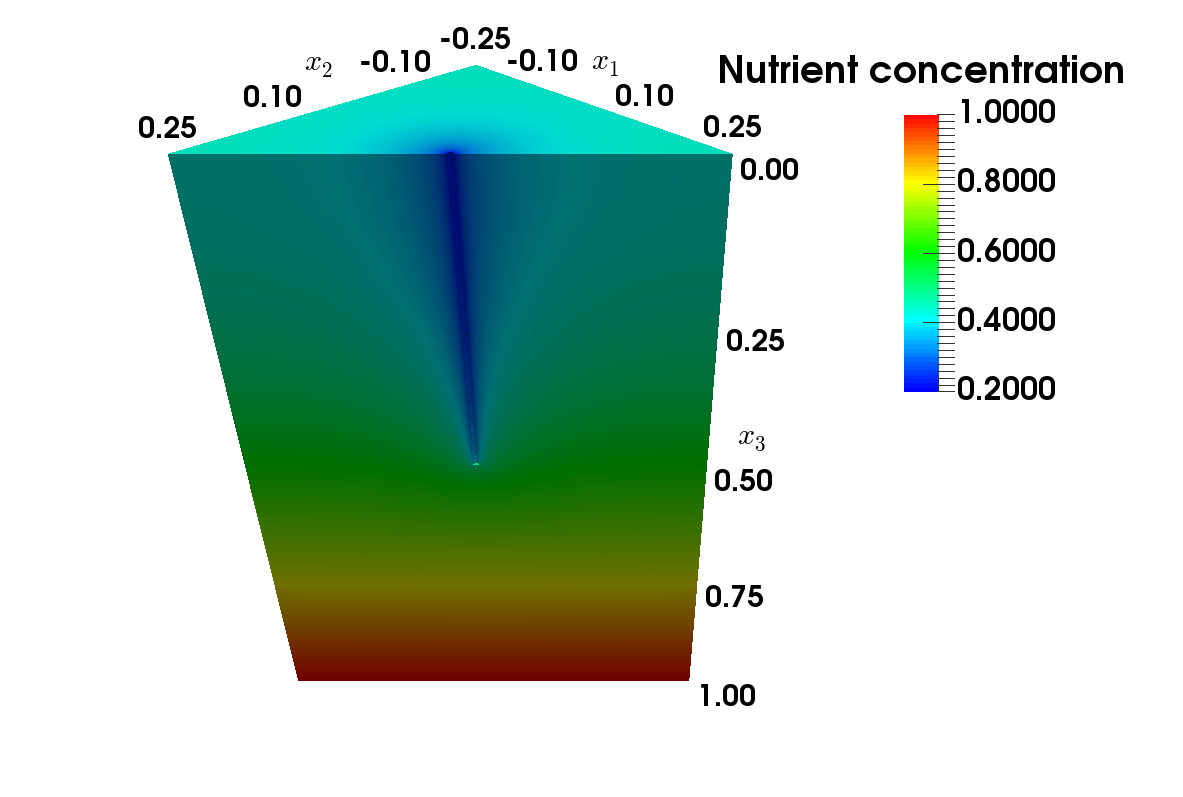}}
\subfloat[$u_\ve$ for $a_\ve=10^{-3}$ ($\ve^{2} \ln{(1/a_\ve)} \approx 1.73$)]{\includegraphics[width=0.5\textwidth]{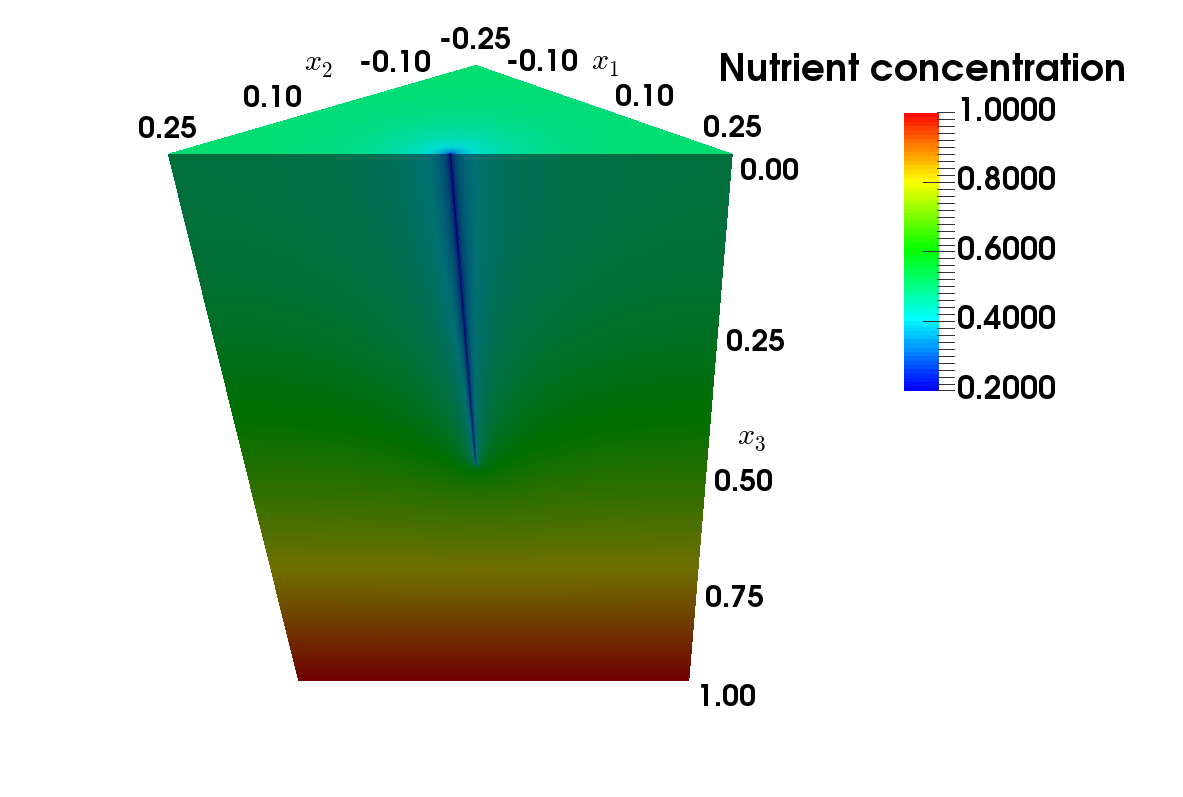}}
\caption{Steady-state solutions of the macroscopic problem \eqref{eqn:first_nonuni_final}, (a), and of the full model \eqref{main_1}--\eqref{bc_hair}, \eqref{bc_rest},  \eqref{ic}, for (b) $a_\ve=10^{-1}$, (c) $a_\ve=10^{-2}$ and (d) $a_\ve=10^{-3}$, with Dirichlet boundary condition \eqref{dirichlet_bc_ff},  $g(u_\ve) = u_\ve$, all other parameters as in Table~\ref{table_1}.}
\label{fig:comparison_dist_limit}
\end{figure}

\begin{table}
\centering
\begin{tabular}{ccccccc}\hline
Parameter & $\ve$ & $L$ & $M$ & $\beta$ & $D_u$ & $\kappa$ \\ \hline
Value & 0.5 & 0.5 & 1.0 & 0.0 & 1.0 & 1.0 \\ \hline \\
\end{tabular}
\caption{Default dimensionless parameter values used in numerical simulations.}
\label{table_1}
\end{table}
For $a_\ve=10^{-1}$ (Figure~\ref{fig:comparison_dist_limit}(b)), the steady-state solution of problem  \eqref{eqn:first_nonuni_final} (Figure~\ref{fig:comparison_dist_limit}(a)) gives a good averaged approximation to that of \eqref{main_1}--\eqref{bc_hair}, \eqref{bc_rest},  \eqref{ic},  whereas for $a_\ve=10^{-2}$ and $a_\ve=10^{-3}$ (Figure~\ref{fig:comparison_dist_limit}(c,d)) the differences between the solution of the macroscopic problem~\eqref{eqn:first_nonuni_final} and those of \eqref{main_1}--\eqref{bc_hair}, \eqref{bc_rest},  \eqref{ic} become more significant and, as  $\ve^{2} \ln{(1/a_\ve)}$ approaches $1$, the steady-state solution of the macroscopic problem \eqref{eqn:homo_dist} provides a better approximation to solutions of the full model, as predicted. The analysis in Section~\ref{sec:sparse_first_nonuni} implies that for any scale relations  satisfying $a_\ve \gg e^{-1/\ve^{2}}$ as $\ve \to 0$ the same macroscopic equation \eqref{eqn:first_nonuni_final} pertains.

We now compare these solutions at a fixed distance from the root surface. First, we fix $x_{3} = 0$ and plot the solutions along a diagonal joining the opposite corners of this plane. This way, we study behavior at the root surface, and the results for decreasing $a_\ve$ are shown in Figure $\ref{fig:comparison_dist_limit_2}$(a,c,e). Solutions of the full problem \eqref{main_1}--\eqref{bc_hair}, \eqref{bc_rest},  \eqref{ic}, (blue) show nutrient depletion zones close to the hair surface with increasingly sharp concentration gradients for a decreasing value of $a_\ve$ due to the scaling of the uptake constant \eqref{scaling_K}. Numerical simulations reveal that the steady-state solution of the macroscopic problem \eqref{eqn:first_nonuni_final} underestimates, and that of the macroscopic problem \eqref{eqn:homo_dist} overestimates, the averaged behavior of steady-state solutions of the full problem \eqref{main_1}--\eqref{bc_hair}, \eqref{bc_rest},  \eqref{ic}. While the solution of \eqref{eqn:first_nonuni_final} provides us with a better approximation to the full-geometry behaviour than that of \eqref{eqn:homo_dist} for $a_\ve = 10^{-1}$, the opposite is true for $a_\ve = 10^{-3}$, which confirms the validity of our asymptotic analysis results. Leading-order approximations (i.e. homogenized solutions) naturally cannot capture large depletion gradients present in full-geometry simulations near root hair surfaces. Comparison with higher-order approximations will be discussed later (see Figure~\ref{fig:comparison_lin_correctors}).

\begin{figure}
\centering
\subfloat[][\centering $x_3 = 0.0$, $a_\ve = 10^{-1}$]{\includegraphics[width=0.48\textwidth]{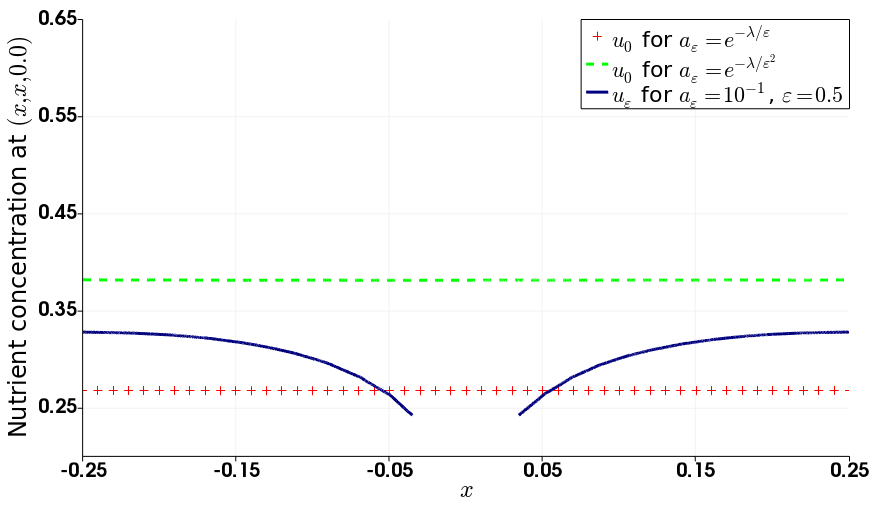}} 
\subfloat[][\centering $x_3 = 0.75$, $a_\ve = 10^{-1}$]{\includegraphics[width=0.48\textwidth]{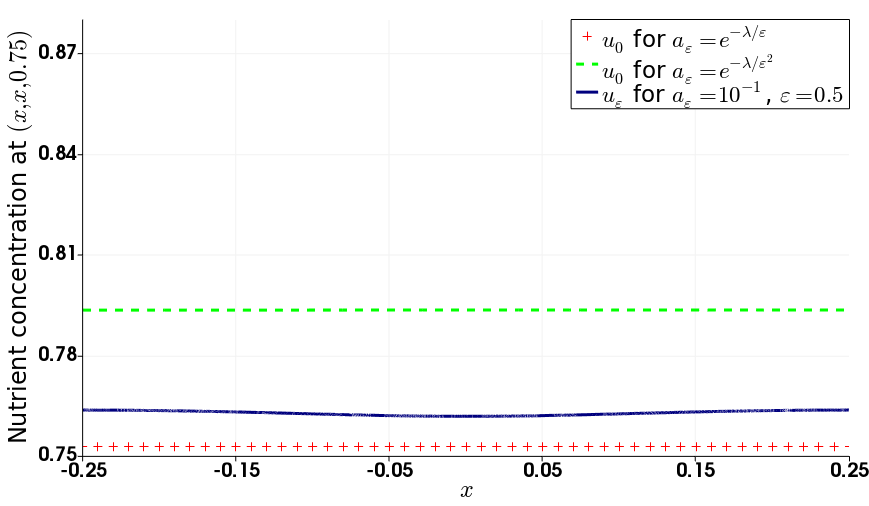}}
\\
\subfloat[][\centering $x_3 = 0.0$, $a_\ve = 10^{-2}$]{\includegraphics[width=0.48\textwidth]{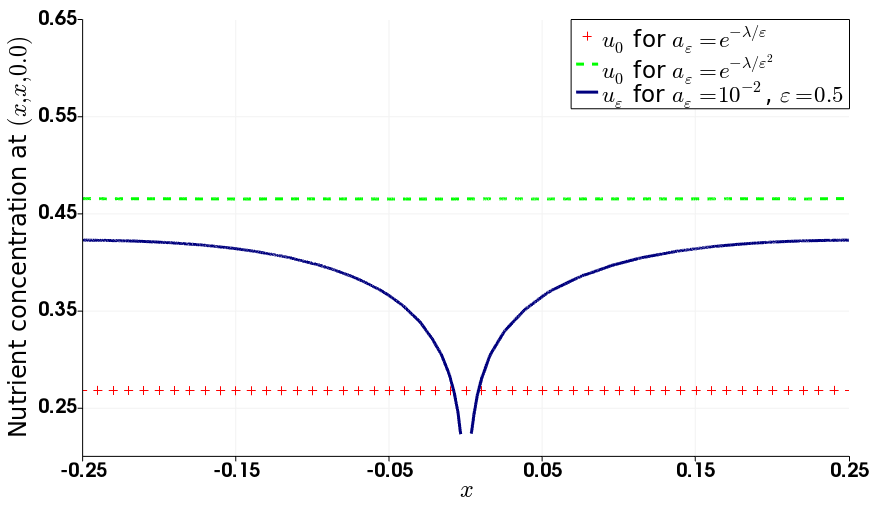}}
\subfloat[][\centering $x_3 = 0.75$, $a_\ve = 10^{-2}$]{\includegraphics[width=0.48\textwidth]{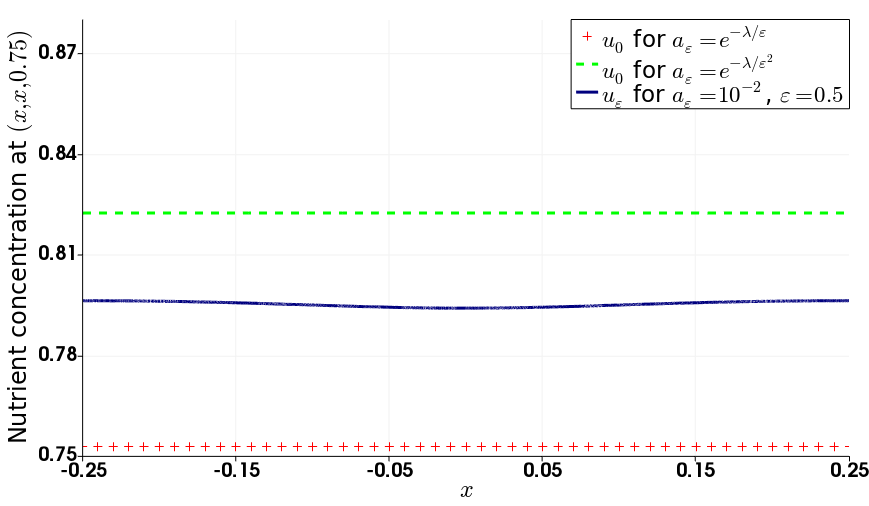}}
\\
\subfloat[][\centering $x_3 = 0.0$, $a_\ve = 10^{-3}$]{\includegraphics[width=0.48\textwidth]{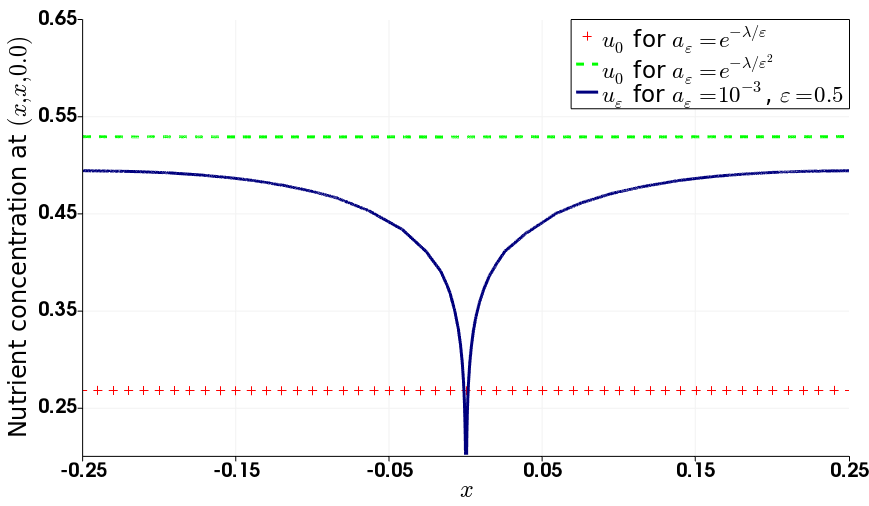}}
\subfloat[][\centering $x_3 = 0.75$, $a_\ve = 10^{-3}$]{\includegraphics[width=0.48\textwidth]{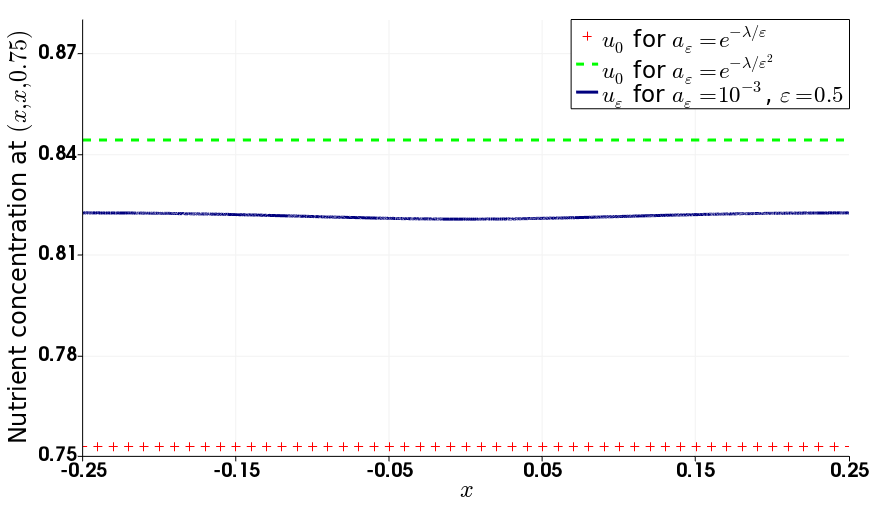}}
\\
\subfloat[][\centering $x_3 = 0.0$, $a_\ve = 10^{-3}$, $g(u) = u/(1+u)$]{\includegraphics[width=0.48\textwidth]{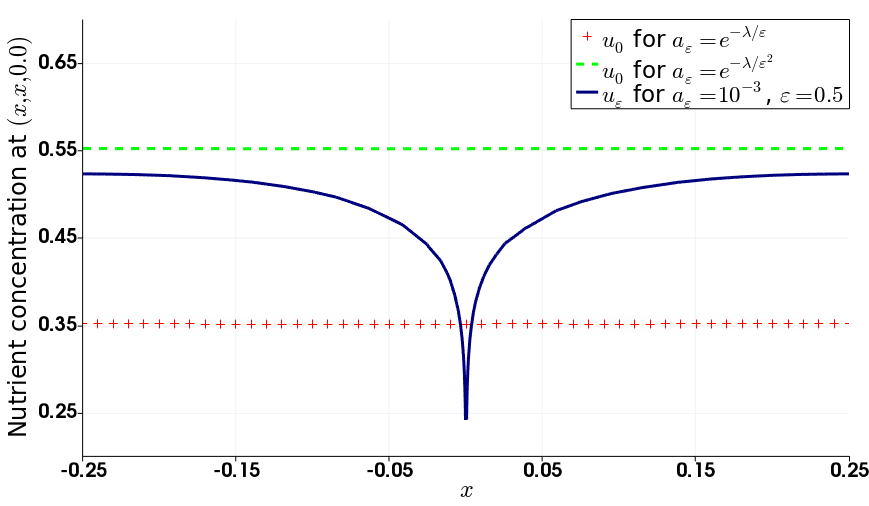}}
\subfloat[][\centering $x_3 = 0.75$, $a_\ve = 10^{-3}$, $g(u) = u/(1+u)$]{\includegraphics[width=0.48\textwidth]{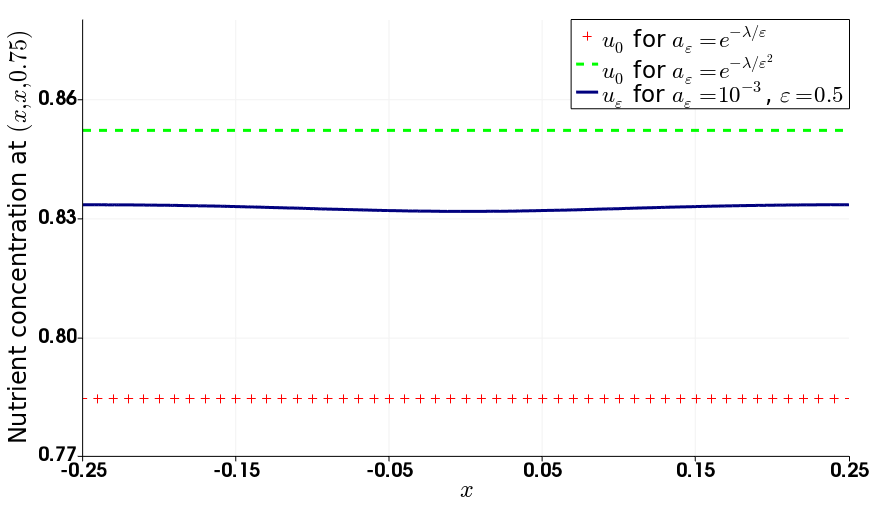}}
\caption{Steady-state solutions at the root surface $\{x_3=0\}$ (figures (a), (c) and (e)) and outside of the root-hair zone $\{x_3=0.75\}$ (figures (b), (d) and (f)) for \eqref{main_1}--\eqref{bc_hair}, \eqref{bc_rest},  \eqref{ic} (blue solid line), the problem \eqref{eqn:first_nonuni_final} (red crosses) and the problem \eqref{eqn:homo_dist} (green dashed line), with boundary condition~\eqref{dirichlet_bc_ff}, $g(u) = u$,  and all other parameters as in Table~\ref{table_1}. $a_\ve$ is decreased from $10^{-1}$ to $10^{-3}$. Figures (g) and (h) show comparisons for the nonlinear problem (with $g(u) = u/(1+u)$) to the problem \eqref{eqn:homo_nonlin}  (green dashed line; for the full form of the continuity equation, see \eqref{eqn:MM_dist_CE_explicit}), and the problem \eqref{eqn:first_nonlin_1}  (red crosses), using the same parameters and boundary conditions.}
\label{fig:comparison_dist_limit_2}
\end{figure}
\label{fig:comparison_MM_SS}

Simulation results at $x_{3} = 0.75$, i.e.\ outside the root hair-zone, see  Figure~\ref{fig:comparison_dist_limit_2}(b,d,f), demonstrate that as $a_\ve$ decreases and approaches the scale relation  
$\ve^2 \ln{(1/a_\ve)} = O(1)$, the steady-state solution of the macroscopic model \eqref{eqn:homo_dist} provides a better approximation to the full model~\eqref{main_1}--\eqref{bc_hair}, \eqref{bc_rest},  \eqref{ic} than that of~\eqref{eqn:first_nonuni_final}.

Numerical solutions to the steady-state problem for \eqref{main_1}--\eqref{bc_hair}, \eqref{bc_rest},  \eqref{ic} with a nonlinear boundary condition on $\Gamma^\ve$, i.e.\ with $g(u_\ve) = u_\ve/ (1+ u_\ve)$, and to the corresponding macroscopic problems \eqref{eqn:first_nonlin_1} and  \eqref{eqn:homo_nonlin} are also presented in Figure~\ref{fig:comparison_dist_limit_2}(g,h).  All model parameters are as in Table~\ref{table_1} and Picard iteration was used to solve the nonlinear problem (as described in \cite{logg2011}). Similar differences between solutions of the full model and the two macroscopic problems are observed in time-dependent solutions, see Figure~\ref{fig:comparison_MM_time} (note that we used a zero-flux boundary condition at $x_3 = M$ in this case,  modelling competition with a neighboring root at $x_3 = 2M$).

\begin{figure}
\subfloat[][\centering $x_3 = 0.75$, $g(u) = u/(1+u)$]{\includegraphics[width=0.49\textwidth]{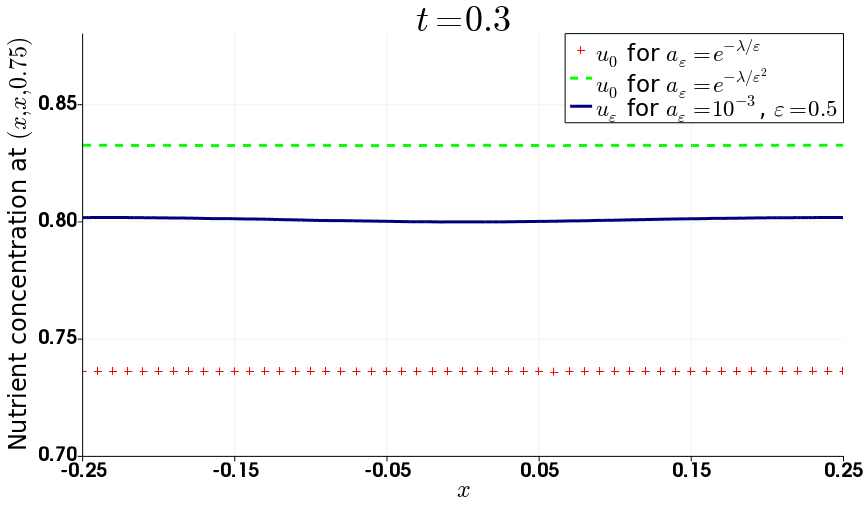}}
\subfloat[][\centering $x_3 = 0.75$, $g(u) = u/(1+u)$]{\includegraphics[width=0.49\textwidth]{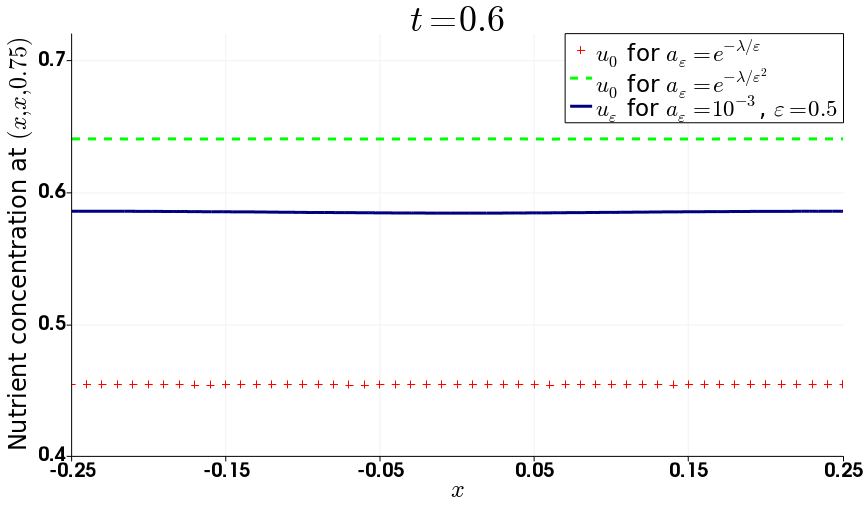}}
\\
\subfloat[][\centering $x_3 = 0.0$, $g(u) = u/(1+u)$]{\includegraphics[width=0.49\textwidth]{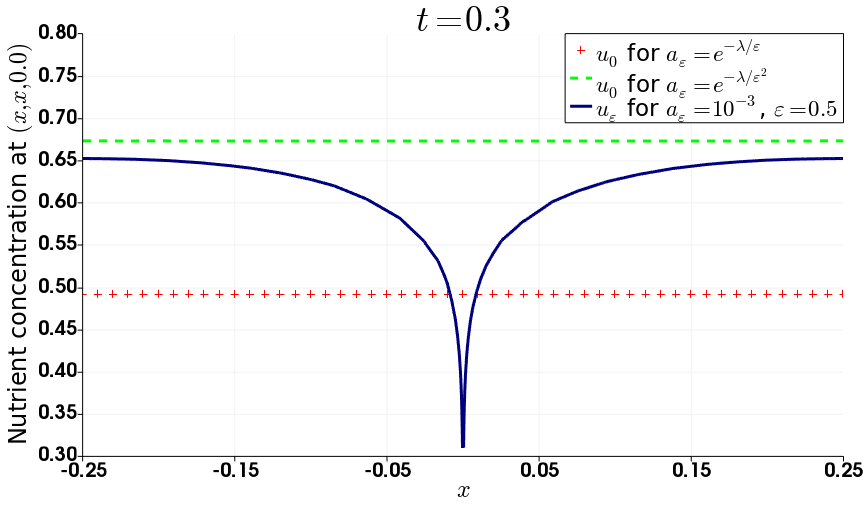}}
\subfloat[][\centering $x_3 = 0.0$, $g(u) = u/(1+u)$]{\includegraphics[width=0.49\textwidth]{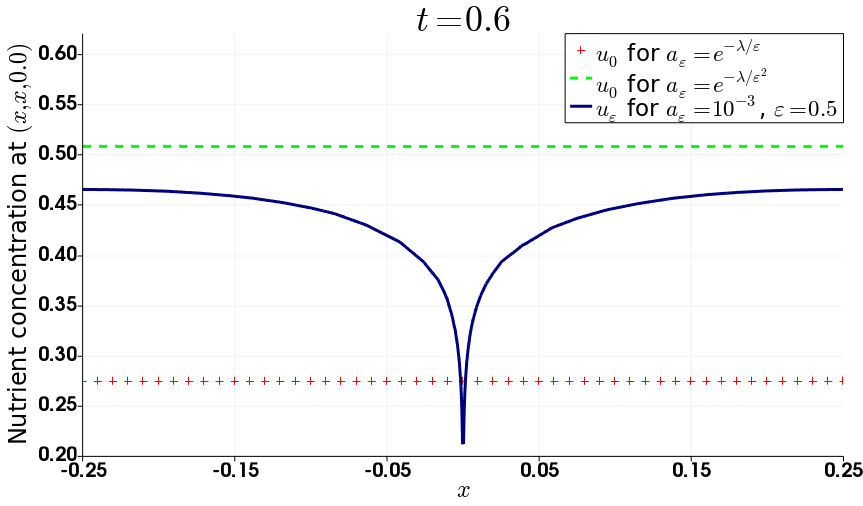}}
\caption{Numerical solutions for \eqref{main_1}--\eqref{bc_hair}, \eqref{bc_rest},  \eqref{ic} (blue solid line), the problem \eqref{eqn:homo_nonlin}  (green dashed line; for the full form of the continuity equation, see \eqref{eqn:MM_dist_CE_explicit}) and the problem \eqref{eqn:first_nonlin_1} (red crosses),  with $g(u) = u/(1+u)$ (figures (a), (b), (c) and (d)), and initial condition   $u_{\rm in} = 1$, all other parameters as in Table~\ref{table_1}. The time derivative is discretized using the backwards Euler method, with the time step of $0.01$.}
\label{fig:comparison_MM_time}
\end{figure}

Numerical solutions for the first and second order corrections, given by \eqref{eqn:first_nonuni_final_2}, \eqref{u_2_first_scale}, \eqref{eqn:homo_dist_1} and \eqref{u_2_second_scale}, for the two different scale relations between $\ve$ and $a_\ve$ are presented in  Figure~\ref{fig:comparison_lin_correctors}.
\begin{figure}
\subfloat[][\centering $x_3 = 0.0$, $g(u) = u$, correctors]{\includegraphics[width=0.49\textwidth]{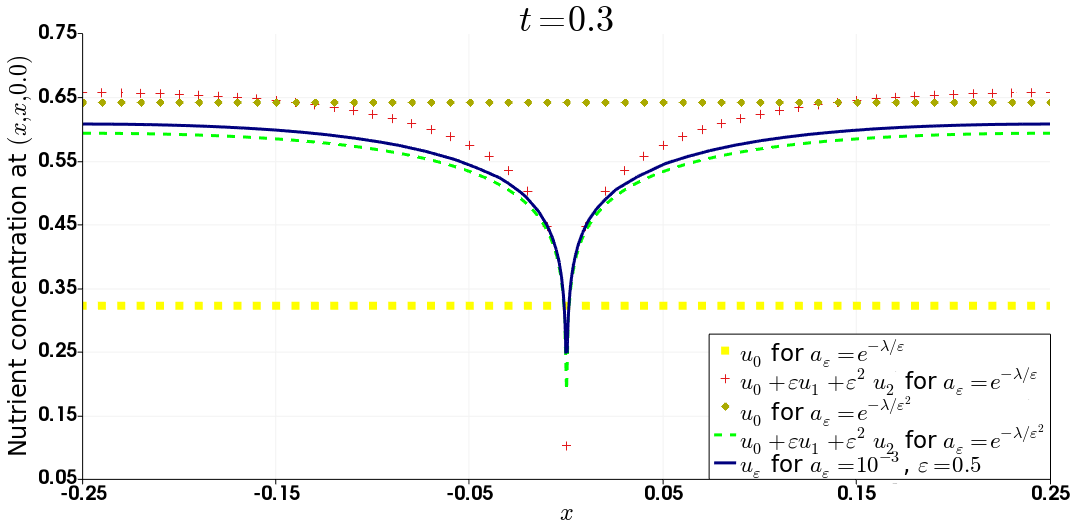}} 
\subfloat[][\centering $x_3 = 0.0$, $g(u) = u$, correctors]{\includegraphics[width=0.49\textwidth]{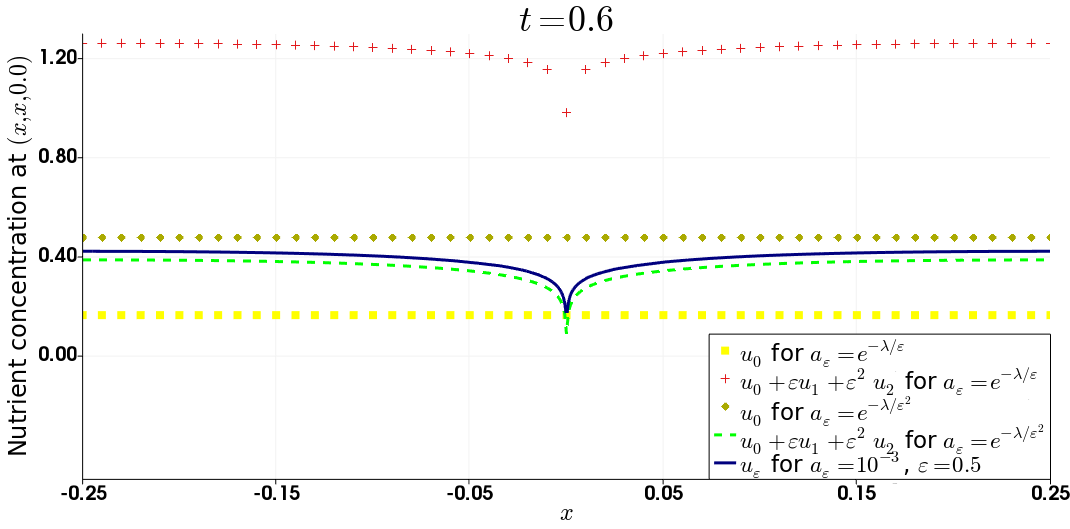}}
\caption{Figures (a) and (b) show comparison at the root surface $\{x_3=0\}$ for the linear problem \eqref{main_1}--\eqref{bc_hair}, \eqref{bc_rest},  \eqref{ic} (blue solid line) with the problem \eqref{eqn:homo_dist}  (brown diamonds), the problem  \eqref{eqn:first_nonuni_final}   (yellow squares), the second-order approximation
\eqref{eqn:first_nonuni_final} - \eqref{u_2_first_scale}  (red crosses), and with the second-order approximation \eqref{eqn:homo_dist} - \eqref{u_2_second_scale} (green dashed line), using the same initial condition and parameters as in Figure \ref{fig:comparison_MM_time}.}
\label{fig:comparison_lin_correctors}
\end{figure}
The differences between these illustrate the importance of the correct approximation. Since we chose our parameters so that $\ve^{2} \ln{\big(1/a_\ve \big)} = O(1)$ we have that solutions of \eqref{eqn:homo_dist}-\eqref{u_2_second_scale} provide better approximations to those of the full problem \eqref{main_1}--\eqref{bc_hair}, \eqref{bc_rest},  \eqref{ic} than solutions of \eqref{eqn:first_nonuni_final}-\eqref{u_2_first_scale}.

\section{Discussion}
\label{sec:discussion}
The analysis in Section~\ref{sec:independent_a_small} using two independent small parameters $\ve$ and $a$ uncovered the term $\ve^2  \ln(1/a) u_{0,0}(t,x) \psi_{-1}^O$, which causes problems relating to commutation of the two limits under consideration (see~\eqref{final_outer_separation}). Based on this observation, we then studied two scale relations given by $\ve \ln(1/a_\ve) = O(1)$ and $\ve^2 \ln(1/a_\ve) = O(1)$. In the $\ve \ln(1/a_\ve) = O(1)$ case, the mentioned term becomes $O(\ve)$, and thus it does not affect the leading-order problem~\eqref{eqn:first_nonuni_final}, but the $O(\ve)$ problem~\eqref{eqn:first_nonuni_final_2}. In the $\ve^2 \ln(1/a_\ve) = O(1)$ case, the same term becomes $O(1)$, affects the leading-order problems and thus leads to distinguished limits, see~\eqref{eqn:homo_dist} for the linear boundary condition and~\eqref{eqn:homo_nonlin} for the nonlinear boundary condition. Notice that the sink term in the distinguished limit~\eqref{eqn:homo_dist} is obtained by dividing the sink term in the standard limit~\eqref{eqn:first_nonuni_final} by $1+ \lambda \kappa / D_u > 1$, implying weaker effective nutrient uptake in the hair zone. This is because assuming $\ve^2 \ln(1/a_\ve) = O(1)$, the uptake rate per unit hair surface area becomes large, causing very sharp nutrient depletion near hairs so that the diffusion is not fast enough to keep the concentration profile uniform. Under these circumstances, the difference between the nutrient concentration at the hair surface (used in the full-geometry model) and the averaged nutrient concentration (used in the sink terms) becomes significant and this gives rise to the new limit. Subsequently, we rigorously proved the convergence of solutions of the multiscale problem to solutions of  the macroscopic equations for both the linear and nonlinear boundary conditions at surfaces of root hairs and confirmed the applicability of the two limit equations (as well as higher-order correctors) in different parameter regimes via numerical simulations.


\newpage

\section*{Supplementary materials}

\subsection*{Parameter values}
The scaling in the boundary conditions on $\Gamma^\ve$ should be interpreted in terms of the experimental values for nutrient uptake rates by root hairs for different plant types.
Considering  the nondimensionalization of dimensional Michaelis-Menten boundary condition
$$ - D \nabla u \cdot \hat{\bf n } = \frac{F_{h} u}{K_{h} + u} $$
via $ x = R \widetilde x$, $t = R^2 \widetilde{t}/D $,   $ u = K_{h} \widetilde{u}$  gives
\begin{equation} \label{eqn:bc_hair_alpha_beta_scalings}
- \widetilde{\nabla} \widetilde{u}  \cdot  {\bf n } = \frac{F_{h} R}{K_{h} D} \frac{\widetilde u }{1+\widetilde u } = \frac{\ve}{a_{\ve}} \frac{r_{h} R^2 F_h}{K_h l^2 D} \frac{\widetilde u }{1+\widetilde u} = \frac{\ve}{a_{\ve}} \widetilde \alpha \frac{\widetilde u }{1+\widetilde u},
\end{equation}
where $r_{h}$ denotes the dimensional hair radius, $l$ denotes the dimensional inter-hair distance and $\widetilde \alpha = ( r_{h} R^2 F_h)/(K_h l^2 D)$. Considering the range of phosphate uptake parameters $F_h$ and $K_h$ as reviewed in \cite{leitner2010_2}, and $D = 10^{-5}$ cm$^2$ s$^{-1}$ \cite{nobel}, as well as  $R = 1$ cm, $l=0.01$ cm and $r_{h} \sim 10^{-4}$ cm, we conclude that $\widetilde \alpha = 10$ for wheat, while $\widetilde \alpha = 1$ arises when modelling sulphur and magnesium uptake by maize \cite{roose2000}.

\subsection*{Derivation of macroscopic equations for nonlinear boundary conditions on root hair surfaces}

\subsubsection*{Case $\ve \ln{(1/a_\ve)} = O(1)$}
Following the same procedure as in Section~\ref{sec:sparse_first_nonuni} of the main text, we obtain the same equations as in \eqref{eq_inner_22}, but with different boundary conditions for
$u_2^I$, $u_3^I$, and $u_4^I$, namely
\begin{equation}\label{eq_inner_nonlin_1}
\begin{aligned}
 &  D_u   \nabla_{z} u_{2}^{I} \cdot \hat{\bf n } = -\kappa g(u_{0}^{I})\;  \;  \text{ on }   \partial B_1, \quad   D_u   \nabla_{z} u_{3}^{I} \cdot \hat{\bf n } = -\kappa g^\prime(u_{0}^{I}) u_1^I\;  \;  \text{ on }   \partial B_1, \\
   &  D_u   \nabla_{z} u_{4}^{I} \cdot \hat{\bf n } =   - \kappa \big[g^\prime(u_0^I) u_2^I+ \frac 12  g^{\prime\prime}(u_{0}^{I})  (u_1^I)^2\big]  \quad   \text{ on }   \partial B_1.
 \end{aligned}
  \end{equation}
  Hence the corresponding solutions are
  $$
\begin{aligned}
& u_j^I(t,x,z) = u_j^I(t,x), \; \; \; j=0,1, \quad 
 u_{2}^{I}(t,x,z) = (\kappa/D_u) g(u_{0}^{I}) \ln{(\|z\|)} + U^I_2(t,x), \\
& u_{3}^{I}(t,x,z) = (\kappa/D_u) g^\prime(u_{0}^{I}) u_1^I \ln{(\|z\|)} + U_3^I(t,x),\\
& u_{4}^{I}(t,x,z) = (\kappa/D_u) \big[g^\prime(u_0^I) U_2^{I}(t,x) + \frac 12  g^{\prime\prime}(u_{0}^{I})  (u_1^I)^2\big]   \ln{(\|z\|)} + U_4^I(t,x).
\end{aligned}
$$
Then by matching inner approximation  $u_2^I$ and outer approximation $u_2^O$ we obtain for $u_2^O$ equation \eqref{eqn:first_nonuni} with $g(u_0^I)$ instead of $u_0^I$ and  for $u_0^O$  equation \eqref{eqn:first_nonuni_OI} with  $g(u_0^I)$ instead of $u_0^I$.
We also obtain the same matching condition \eqref{matching_limit_1}.
Hence we obtain an effective equation
\begin{equation} \label{eqn:first_nonlin_1_CE_only}
\begin{aligned}
&\partial_t  u_{0} = \nabla_{x}\cdot(D_u \nabla_x u_{0})- 2 \pi \kappa \, g(u_{0})\,  \chi_{\Omega_L}&& \text{ in } \Omega, \; t>0,
\end{aligned}
\end{equation}

\subsubsection*{Case $\ve^2 \ln{(1/a_\ve)} = O(1)$}
Applying the formal asymptotic expansion ansatz \eqref{eqn:ansatz_dist} in multiscale problem \eqref{main_1}--\eqref{bc_hair}, \eqref{bc_rest},  \eqref{ic} again yields~\eqref{outer_22_nextscaling}, equipped here with the modified boundary condition
$$
\begin{aligned}
& \left( e^{{\lambda}/{\ve^{2}}}  \ve^{-1} D_u \nabla_{z} + D_u \nabla_{x} \right) \left( u_0 + \ve u_{1}  + \cdots \right)  \cdot \hat{\bf n }  = -\ve \,  \kappa \, e^{\lambda/\ve^2} g\big(u_{0}+ \ve u_1 + \cdots \big)\\
 &  = -\ve \,  \kappa \, e^{\lambda/\ve^2} \big[g(u_{0}) + \ve  g^\prime(u_0) u_1 + \ve^2  g^\prime(u_0) u_2 +  \ve^2 \frac 12  g^{\prime\prime} (u_0) u_1^2+ \cdots \big]   \; \; \text{ on }  \Omega_L \times \partial B_1.
\end{aligned}
$$
In the case of  inner solutions,     for $u_0^I$ and $u_1^I$ we have  the same equations  and boundary conditions as in \eqref{eq_inner_22}  and for $u_2^I$,  $u_3^I$, and $u_4^I$ we obtain the same equations as in \eqref{eq_inner_22} but with different boundary conditions 
\begin{equation}\label{eq_inner_nonlin}
\begin{aligned}
 &  D_u   \nabla_{z} u_{2}^{I} \cdot \hat{\bf n } = -\kappa g(u_{0}^{I})  \quad  && \text{ on }   \partial B_1, \qquad \\
 & D_u   \nabla_{z} u_{3}^{I} \cdot \hat{\bf n } = -\kappa g^\prime(u_{0}^{I}) u_1^I \quad && \text{ on }   \partial B_1, \\
    &  D_u   \nabla_{z} u_{4}^{I} \cdot \hat{\bf n } =   - \kappa \big[g^\prime(u_0^I) u_2^I+ \frac 12  g^{\prime\prime}(u_{0}^{I})  (u_1^I)^2\big]  \quad   && \text{ on }   \partial B_1.
 \end{aligned}
  \end{equation}
  Hence the inner approximation reads
\begin{equation}
\begin{aligned}
u_\ve^{I}(t,x) & =  u_0^I(t,x) + \ve  u_1^I(t,x) + \ve^2 U_2^I(t,x) + \ve^2 (\kappa/D_u) g(u_0^I) \ln{(\|z\|)} \\
& + \ve^3 \Big[(\kappa/D_u) g^\prime(u_0^I)\, u_1^I \ln{(\|z\|)} + U_3^I(t,x)\Big] \\
& +
 \ve^4 \Big[\frac{\kappa}{D_u} \big(g^\prime(u_0^I)\, U_2^I  + \frac 12  g^{\prime\prime}(u_{0}^{I}) (u_1^I)^2\big)\ln{(\|z\|)} + U_4^I(t,x)\Big] +\cdots.
 \end{aligned}
\end{equation}
Then in terms of outer variables $y$ the inner approximation $u_\ve^{I}$ has the form
\begin{eqnarray}
u_\ve^{I}  
&& = \Big( u_0^I + \lambda \frac{\kappa}{D_u} g(u_0^I) \Big) + \ve  \Big( u_1^I + \lambda \frac{\kappa}{D_u} g^\prime(u_0^I) u_1^I \Big) \nonumber \\
&& \qquad \quad+ \ve^2 \big[U_2^I + \frac{\kappa}{D_u} g(u_0^I) \ln{(\|y\|)} + \lambda \frac{\kappa}{D_u}\big( g^\prime(u_0^I) U_2^I +  \frac 12  g^{\prime\prime}(u_{0}^{I}) (u_1^I)^2\big)  \big] + \cdots. \nonumber
\end{eqnarray}
In the same way  as in   Subsection~\ref{sec:sparse_next_nonuni}, for the outer approximation  we obtain
\begin{equation*}
\begin{aligned}
 u_\ve^{O}(t,x) = u_0^O(t,x) + \ve u_1^O(t,x) + \ve^2 \Big(  U_2^O(t,x) + 2 \pi (\kappa/D_u)  g(u_{0}^I(t,x)) \psi(y) \Big)+ \cdots. 
 \end{aligned}
\end{equation*}
Then the matching condition for inner and outer solutions  for zero order terms implies
\begin{equation}\label{nonlin_relat_0}
 u_0^O(t,x)  =   u_0^I(t,x) + \lambda (\kappa/D_u)  g(u_0^I(t,x)),
\end{equation}
 and the macroscopic equation for $u_0(t,x) = u_0^O(t,x)$ reads
 \begin{equation} \label{eqn:homo_nonlin_CE_only}
\begin{aligned}
&\partial_t u_{0}  = \nabla_{x}\cdot(D_u \nabla_x u_{0}) - 2\pi \kappa \,  g(h(u_{0})) \chi_{\Omega_L} && \text{ in } \Omega,  \; t>0,
\end{aligned}
\end{equation}
where $h=h(u_0)$ is the solution of  $u_0 =h  + \lambda \,  (\kappa/D_u) g(h)$.

Adopting the Michaelis-Menten boundary condition \eqref{eqn:MM_form},  condition \eqref{nonlin_relat_0} can be rewritten as a quadratic equation
\begin{equation}\label{eq_quad_MM}
 (u_0^{I})^2 + u_0^I \big( \lambda (\kappa/D_u) + 1 - u_0^O \big) - u_0^O =0,
  \end{equation}
with unique non-negative solution
$$u_{0}^I = \frac 12 \Big[ \sqrt{(u_{0}^O- \lambda (\kappa/D_u) - 1)^2 + 4 u_0^O} + u_{0}^O - \lambda \frac{\kappa}{D_u} - 1 \,  \Big],$$
and the effective equation \eqref{eqn:homo_nonlin_CE_only} thus becomes \eqref{eqn:MM_dist_CE_explicit}.

\begin{figure}[ht]
\centering
\begin{overpic}[width=0.6\textwidth,tics=10]{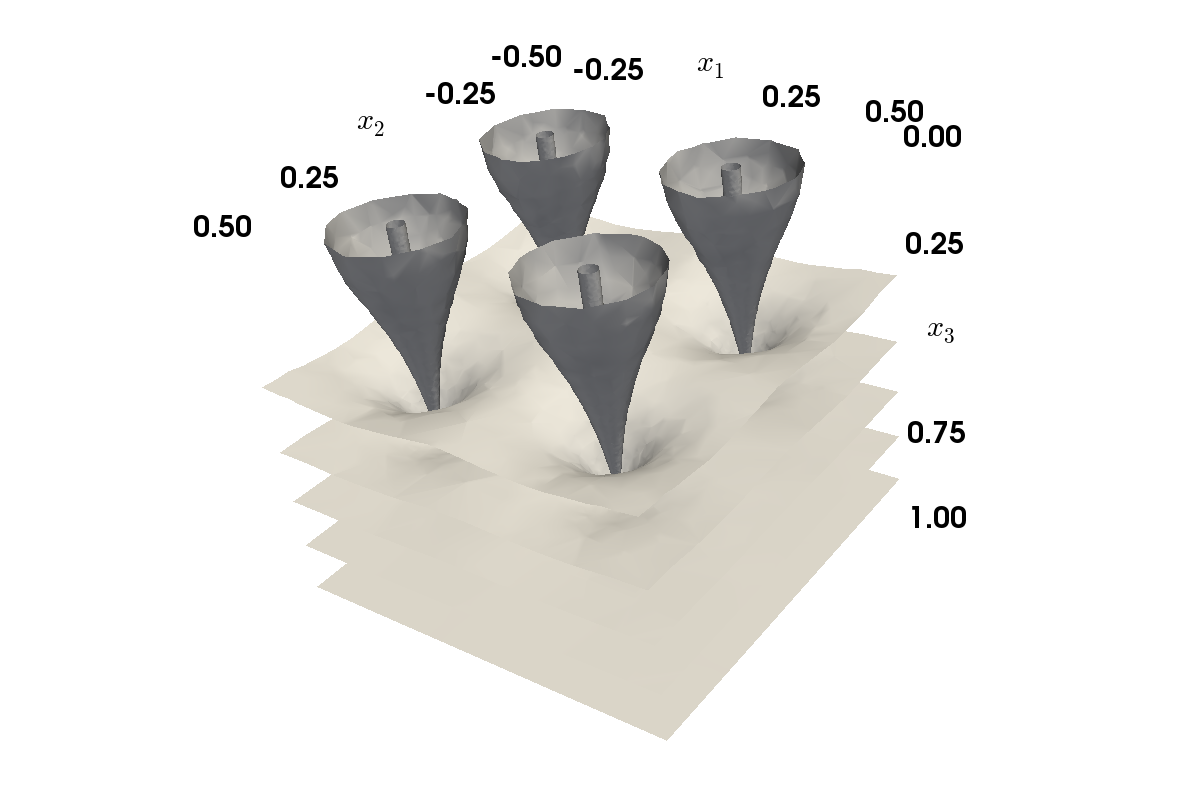}
\linethickness{1pt}
\color{black}
\put(185,130){\vector(0,-1){90}}
\end{overpic}
\caption{Isosurfaces of nutrient concentration support the intuition that with the chosen boundary conditions, the (steady-state) solution has the same behavior in every periodicity cell ($a_\ve=0.01$, $\ve = 0.5$). The arrow points in the direction of increasing $x_3$ (i.e. away from the root surface located at $x_3 = 0$).}
  \label{fig:isosurfaces_eps_05}
\end{figure}

\newpage

\bibliographystyle{plain}
\bibliography{sample_capital}

\end{document}